\documentclass[12pt]{amsart}%
\usepackage{amssymb,amsfonts}
\usepackage{amscd}
\usepackage{amsmath}
\usepackage{amsfonts}
\usepackage{amssymb}
\usepackage{verbatim}
\usepackage{mathrsfs}
\usepackage{bm}
\usepackage{comment}
\usepackage{tikz-cd}
\usepackage{xcolor}
\usepackage{hyperref}

\hypersetup{
    colorlinks=true,
    citecolor=blue, 
     linkcolor=black,     
    urlcolor=magenta, 
}

\providecommand{\U}[1]{\protect\rule{.1in}{.1in}}
\newtheorem{theorem}{Theorem}[section]
\newtheorem{lemma}[theorem]{Lemma}
\newtheorem{proposition}[theorem]{Proposition}

\newtheorem{example}[theorem]{Example}

\newtheorem{corollary}[theorem]{Corollary}
\newtheorem{remark}[theorem]{Remark}

\numberwithin{equation}{section}

\newcommand{\cB}{\mathcal{B}}

\newcommand{\cE}{\mathcal{E}}

\newcommand{\cF}{\mathcal{F}}

\newcommand{\cH}{\mathcal{H}}

\newcommand{\cK}{\mathcal{K}}

\newcommand{\cV}{\mathcal{V}}

\newcommand{\bC}{\mathbb{C}}

\newcommand{\bB}{\mathbb{B}}

\newcommand{\bM}{\mathbb{M}}

\newcommand{\bN}{\mathbb{N}}

\newcommand{\bW}{\mathbb{W}}

\newcommand{\fA}{{\mathfrak{A}}}
\newcommand{\fB}{{\mathfrak{B}}}
\newcommand{\fC}{{\mathfrak{C}}}

\newcommand{\bea}{\begin{eqnarray}}
\newcommand{\eea}{\end{eqnarray}}

\newcommand{\cla}{\mathcal{A}}
\newcommand{\clb}{\mathcal{B}}

\newcommand{\cld}{\mathcal{D}}
\newcommand{\cle}{\mathcal{E}}
\newcommand{\clf}{\mathcal{F}}

\newcommand{\clh}{\mathcal{H}}
\newcommand{\clk}{\mathcal{K}}

\newcommand{\clq}{\mathcal{Q}}

\newcommand{\cls}{\mathcal{S}}

\newcommand{\D}{\mathbb{D}}
\newcommand{\N}{\mathbb{N}}
\newcommand{\C}{\mathbb{C}}

\newcommand{\z}{\bm{z}}
\newcommand{\w}{\bm{w}}

\def\textmatrix#1&#2\\#3&#4\\{\bigl({#1 \atop #3}\ {#2 \atop #4}\bigr)}
\def\dispmatrix#1&#2\\#3&#4\\{\left({#1 \atop #3}\ {#2 \atop #4}\right)}
\newcommand{\be}{\begin{equation}}
\newcommand{\ee}{\end{equation}}
\newcommand{\ben}{\begin{eqnarray*}}
\newcommand{\een}{\end{eqnarray*}}

\newcommand{\bi}{\begin{itemize}}
\newcommand{\ei}{\end{itemize}}

\newcommand\la{{\langle }}
\newcommand\ra{{\rangle}}

\newtheorem{Theorem}{\sc Theorem}[section]
\newtheorem{Lemma}[Theorem]{\sc Lemma}
\newtheorem{Proposition}[Theorem]{\sc Proposition}
\newtheorem{Corollary}[Theorem]{\sc Corollary}
\newtheorem{Definition}[Theorem]{\sc Definition}
\newtheorem{Example}[Theorem]{\sc Example}

\newtheorem{Remark}[Theorem]{\sc Remark}

\newtheorem{Note}[Theorem]{\sc Note}
\newtheorem{Question}{\sc Question}
\newtheorem{ass}[Theorem]{\sc Assumption}
\newcommand{\bt}{\begin{Theorem}}
\def\beginlem{\begin{Lemma}}
\def\beginprop{\begin{Proposition}}
\def\begincor{\begin{Corollary}}
\def\begindef{\begin{Definition}}
\def\beginexamp{\begin{Example}}
\def\beginrem{\begin{Remark}}
\def\beginq{\begin{Question}}
\def\beginass{\begin{ass}}
\def\beginnote{\begin{Note}}
\newcommand{\et}{\end{Theorem}}
\def\endlem{\end{Lemma}}
\def\endprop{\end{Proposition}}
\def\endcor{\end{Corollary}}
\def\enddef{\end{Definition}}
\def\endexamp{\end{Example}}
\def\endrem{\end{Remark}}
\def\endq{\end{Question}}
\def\endass{\end{ass}}

\makeatother

\begin{document}

\title[nc commutant lifting, interpolation, and shift-invariant subspaces]{noncommutative commutant lifting, interpolation, and shift-invariant subspaces}

\author[Barik]{Sibaprasad Barik}
\address{Statistics and Mathematics Unit, Indian Statistical Institute, 8th Mile, Mysore Road, Bengaluru, 560059, India}

\email{barik\_if@isibang.ac.in, sibaprasadbarik00@gmail.com}

	\author[Panja]{Samir Panja}
\address{Statistics and Mathematics Unit, Indian Statistical Institute, 8th Mile, Mysore Road, Bengaluru, 560059, India} 
\email{samirpanja\_pd@isibang.ac.in, panjasamir2020@gmail.com}

	\begin{abstract}
We establish a Sarason-type commutant lifting theorem in the framework of noncommutative Drury-Arveson space and apply it to solve the noncommutative Nevanlinna-Pick interpolation problem on the noncommutative row ball. In particular, it recovers the classical Nevanlinna-Pick interpolation theorem over the open unit disc.  We also provide a Beurling-Lax-Halmos-type characterization of shift-invariant subspaces of noncommutative Drury-Arveson space using dilation theoretic technique.
	\end{abstract}
    
\maketitle 
\section*{Notation}
\begin{list}{\quad}{}
\item $\mathbb{D}$ \quad \quad \quad\,\,\, \, \,  \ \ Open unit disc in the complex plane
$\mathbb{C}$.

\item  $\mathbb{D}^d$ \; \quad \quad\,\,\, \, \,  \  \ Open unit polydisc in the $d$-dimensional complex plane $\mathbb{C}^d$.

\item $\bB^d$\; \quad\quad\,\ \,\, \, \, \ \ Open unit ball in  $\C^d$.

\item  $\bW_d$\; \, \, \, \, \, \, \, \, \ \ Free monoid on $d$ generators.

\item $M_n(\C)$\; \, \, \, \, \, \,   Set of all $n\times n$-matrices with complex entries.

\item  $\clh$, $\cle$, $\clf$ \; \, \, \, \ \ Hilbert spaces.

\item $\clb(\clh)$ \quad \;\,\,\, \, \, \, \ The space of all bounded linear
operators on $\clh$.

\item  $K$, $\clk$, $\Tilde{K}$, $\Tilde{\clk}$  \; \,  Positive semi-definite kernels.

\item $\fB_d$ \; \; \, \, \, \, \, \, \, \  Open unit noncommutative row ball.

\item $\clh^2_d$\; \; \, \, \, \, \, \, \,  \ \ Noncommutative Drury-Arveson space.



\end{list}

The last two notations are defined explicitly in the paper. Throughout this article, we shall use the abbreviation `\textit{nc}' for `\textit{noncommutative}', and all Hilbert spaces are taken over the complex numbers.

\section{Introduction}
One of the main goals of this article is to explore a connection between Nevanlinna-Pick interpolation and commutant lifting in the noncommutative setup. The classical Nevanlinna-Pick interpolation problem on the open unit disc $\D$ asks the following: given $n$ distinct points $z_1,\ldots,z_n$ in $\D$ and $n$ values $w_1,\ldots,w_n$ in $\D$, whether there exists a bounded holomorphic function $\varphi$ on $\D$ such that, $\|\varphi\|_{\infty}\leq 1$ and $\varphi$ interpolates the data, that is,
\[
\varphi(z_i)=w_i\quad(i=1,\ldots,n).
\]
By the Nevanlinna-Pick theorem (\cite{N, Pick}), the above interpolation problem is solvable if and only if the Pick matrix
\[
\begin{bmatrix}
	\frac{1-w_i\bar w_j}{1-z_i\bar z_j}
\end{bmatrix}_{i,j=1}^n
\]
is positive semi-definite. Here the algebra of bounded holomorphic functions on $\D$ is denoted by $H^{\infty}(\D)$, and a function $\varphi\in H^{\infty}(\D)$ with $\|\varphi\|_{\infty}\leq 1$ is called a Schur class function on $\D$. The algebra $H^{\infty}(\D)$ is the multiplier algebra of the Hardy space $H^2(\D)$, which is a reproducing kernel Hilbert space (\cite{Aron}) corresponding to the kernel $K$ on $\D$, known as the Szeg\H o kernel on $\D$, defined by 
\[
K(z,w)=\frac{1}{1-z\bar w}\quad(z,w\in \D).
\]
In view of this connection, Sarason (\cite{S}) proved a commutant lifting theorem in the setting of $H^2(\D)$, which provides an alternative operator-theoretic approach for solving the classical Nevanlinna-Pick interpolation problem. Recall that Sarason's commutant lifting theorem says the following: if $\clq\subseteq H^2(\D)$ is shift co-invariant and $X\in\clb(\clq)$ with $\|X\|\leq 1$ satisfies $(P_{\clq}M_z|_{\clq})X=X(P_{\clq}M_z|_{\clq})$, where $M_z$ is the shift operator on $H^2(\D)$ (multiplication by the co-ordinate function), then
\[
X^*=M_{\varphi}^*|_{\clq},
\]
for some Schur class function $\varphi\in H^{\infty}(\D)$.

The classical Nevanlinna–Pick interpolation theorem has been extensively generalized
to several multivariable domains in the commutative setting, including the open unit polydisc $\mathbb{D}^{d}$ and the open unit ball $\mathbb{B}^{d}$. A natural framework for studying the interpolation problem on $\D^d$ is provided by the Hardy space $H^2(\D^d)$, which is a reproducing kernel Hilbert space corresponding to the kernel 
\[
K(\z,\w)=\prod_{i=1}^d\frac{1}{1-z_i\bar w_i}\quad(\z=(z_1,\ldots,z_d),\w=(w_1,\ldots,w_d)\in \D^d).
\]
For further details on commutant lifting and interpolation on $\D^d$, we refer \cite{BLTT, BBD, BDS, DasPanja, DS} and the references therein. 
On the other hand, a natural framework for studying the interpolation problem on $\bB^d$ is provided by the Drury-Arveson space, which is a reproducing kernel Hilbert space corresponding to the kernel
\[
K(\z,\w)=\frac{1}{1-\langle\z,\w\rangle}\quad(\z=(z_1,\ldots,z_d),\w=(w_1,\ldots,w_d)\in \bB^d).
\]
For unit ball $\bB^d$, there are several articles (see \cite{AT, BTV, DL, EP} and the references therein), where the authors proved commutant lifting theorems and studied the corresponding interpolation problems.

The study of noncommutative (nc) functions has emerged as a vibrant and rapidly developing area of research, whose roots go back to the foundational works
of Taylor (\cite{T72,T73}) and Voiculescu (\cite{V85,V86,V04,V10}). Theories of noncommutative free analysis have been developed through a series of works, including \cite{AMC15,AME15,AM16,KV, P06}. Beyond its intrinsic theoretical significance, the subject has applications in several areas such as free probability (\cite{PV}) and real and complex algebraic geometry (\cite{H,HM}). Let $$\bM^{d}:=\bigsqcup_{n=1}^{\infty}M_n^d(\C)$$ denote the disjoint union of $d$-tuple of $n\times n$ matrices $M_n^d(\C)$ over all finite orders $n$. A subset $\Omega$ of $\bM^d$ is called an nc set if it is closed under direct sums, and a $d$-variable function on $\Omega$ is said to be an nc function if it is graded and respects direct sums and similarities (see Subsection \ref{Sub: nc function} for a detailed discussion). By equipping an nc set $\Omega$ with a suitable topology, referred to as the free topology, we denote the algebra of all bounded holomorphic functions by $H^{\infty}(\Omega)$ (see Section \ref{Sec: pre}), and we let $H_1^{\infty}(\Omega)$ denote the closed unit ball of $H^{\infty}(\Omega)$. 

The noncommutative Nevanlinna-Pick interpolation problem on a nc domain $\Omega$ asks the following: \textit{for $1\leq i\leq m$, given $Z_i\in \Omega\cap M^d_{n_i}(\C)$ and $W_i\in M_{n_i}(\C)$, whether there exists an nc function $\Phi\in H_1^{\infty}(\Omega)$ such that 
\[
\Phi(Z_i)=W_i\quad (i=1,\ldots,m).
\]}
Since the above interpolation problem on general nc domains are typically difficult to solve, we focus on a particular class of nc domains, namely, the basic free open sets. These domains are defined by matrices $\delta$ whose entries are free polynomials and usually denoted by $G_{\delta}$ (some authors also use the notation $\D_{Q}$, where $Q$ is the matrix of free polynomials). We refer the reader to Subsection \ref{Sub: nc function} for a detailed discussion of these domains. In \cite[Theorem 1.3]{AM15}, the authors established a necessary and sufficient condition for the solvability of the nc Nevanlinna-Pick interpolation problem on the domains $G_{\delta}$. Also, in \cite{BMV18}, the authors considered a more general nc interpolation problem on $G_\delta$, namely, the left-tangential interpolation problem and obtained a necessary and sufficient condition for its solvability (see \cite[Theorem 3.1]{BMV18}). 

In this article, we study nc Nevanlinna-Pick interpolation problem on a particular $G_\delta$ domain in $\bM^d$, namely the open unit nc row ball, denoted by $\fB_d$, defined by
\[
\fB_d:=\Big\{X=(X_1,\ldots,X_d)\in\bM^{d}:\|X\|^2=\Big\|\sum_{j=1}^dX_jX_j^*\Big\| <1\Big\}.
\] 
To solve the nc Nevanlinna-Pick interpolation problem on $\fB_d$, we first establish a Sarason-type commutant lifting theorem in the nc Drury-Arveson space $\clh_d^2$, which is an nc reproducing kernel Hilbert space corresponding to the nc kernel 
\[
   \clk:\fB_d\times\fB_d\to\bigsqcup_{m,n=1}^{\infty} B\big(M_{m\times n}(\bC), M_{m\times n}(\bC)\big) 
    \]
    defined by
    \[
   \clk(Z,W)(T)=\sum_{\alpha\in \bW_d}Z^{\alpha}T(W^{\alpha})^*,
    \]
     for $Z\in\fB_d\cap M_m^d(\C)$, $W\in \fB_d\cap M_n^d(\C)$, and $T\in M_{m\times n}(\bC)$, where $\bW_d$ is the free monoid on $d$ generators, denoted here by the first $d$ natural numbers $\{1,...,d\}$ (see Subsection \ref{Sub: nc Drury} for detailed explanation of the notations). Using the fact that $H^{\infty}(\fB_d)$ is the multiplier algebra of $\clh_d^2$ (as established in \cite{BMV18, SSS}), we obtain, as an application of our commutant lifting theorem,  a necessary and sufficient condition for solvability of the corresponding interpolation problem on $\fB_d$ (see Theorem \ref{Scaler nc pick theorem} below). Although the condition is similar to that in \cite{BMV18} (also see \cite{AJP}), our approach is entirely different from that of \cite{BMV18}, and we follow the classical approach developed by Sarason in \cite{S}. Our approach is more direct and explicit: the analysis in Sections \ref{commutant} and \ref{lifting and interpolation} is carried out entirely within the framework of the nc Drury-Arveson space and it provides a more intrinsic and transparent treatment of the problem. Moreover, our proof of the interpolation theorem via commutant lifting explicitly demonstrates the naturality of the necessary and sufficient condition (see Theorem \ref{commutant_lifting theorem} below). We also show how the nc Nevanlinna-Pick interpolation theorem recovers the classical Nevanlinna–Pick interpolation theorem on $\mathbb{D}$ (see Subsection \ref{subsection_recovery NP theorem} below).

     Another goal of this article is to study the joint shift invariant subspaces of the nc Drury-Arveson space. For a bounded operator $T$ on a Hilbert space $\clh$, the invariant subspace problem asks whether there always exists a non-trivial closed subspace $\cls\subseteq \clh$ such that $T\cls\subseteq \cls$. For infinite-dimensional separable Hilbert spaces, this problem is still open. By Beurling’s theorem (\cite{Beur}), the nontrivial closed $M_z$-invariant subspaces of the Hardy space $H^2(\D)$ are precisely those of the form $\theta H^2(\D)$, where $\theta\in H^\infty(\D)$ is an inner function. Subsequently, Lax (\cite{Lax}) and Halmos (\cite{Halmos}) generalized Beurling's theorem to the vector-valued Hardy space $H_{\cle}^2(\D)$ ($\cong H^2(\D)\otimes \cle$).

They showed that a non-trivial closed subspace $\cls\subseteq H_{\cle}^2(\D)$ is $M_{z}\otimes I_{\cle}$-invariant if and only if there exist an auxiliary Hilbert space $\clf$ and an inner function $\theta\in H_{\clb(\clf,\cle)}^{\infty}(\D)$ such that
\[
\cls=\theta H_{\clf}^2(\D),
\]
where $H_{\clb(\clf,\cle)}^{\infty}(\D)$ is the algebra of $\clb(\clf,\cle)$-valued bounded analytic functions on $\D$.
 
In the setting of the commutative Drury-Arveson space $H_d^2$ over the unit ball $\bB^n$, a non-trivial closed subspace $\cls \subseteq H_d^2$ is invariant under the row shift operator $(M_{z_1},\ldots, M_{z_d})$ if and only if there exist a Hilbert space $\clf$ and a partial isometric multiplier $\Psi\in \text{Mult}(H_d^2\otimes I_{\clf}, H_d^2)$ such that
\[
\cls=\Psi (H_d^2\otimes I_{\clf})
\]
(cf. \cite{GRS, MT, Sarkar-II}). A similar result also holds for general vector-valued analytic reproducing kernel Hilbert spaces $H(K)\otimes \cle$, where $K$ is an analytic kernel on $\D$ or $\bB^d$ (see \cite{Sarkar-I,Sarkar-II, MT}). Results in \cite{Sarkar-I} and \cite{Sarkar-II} were proved using properties of dilation theory and multiplier spaces of reproducing kernel Hilbert spaces. 

In the noncommutative setting, characterization of invariant subspaces  has been investigated by several authors from various perspectives (see \cite{BV22, MS, Po89, P10, P18}).
 We characterize the shift-invariant subspaces of the nc Drury-Arveson space (see Theorem \ref{inv-sub-left-shift}) adapting the techniques from \cite{Sarkar-I} and \cite{Sarkar-II}. Our proof is entirely new in this setting, relies on the dilation result established in Theorem \ref{dilation} below, together with the explicit dilation map. 

The rest of the paper is organized as follows. In the next section, we review some general theory of nc functions, nc kernels, and nc reproducing kernel Hilbert spaces. Section \ref{vector-valued-Drury} is devoted to the discussion  nc Drury-Arveson spaces, both scalar and vector-valued cases, together with a detailed discussion of their multiplier algebras. In this section, we also consider both the left and right multipliers as they are linked through some commutation relations, developed and used in later sections. In Section \ref{commutant}, we characterize the commutants of the shift operators on the nc Drury-Arveson space in a general vector-valued setup. In Section \ref{lifting and interpolation}, we prove a commutant lifting theorem for the nc Drury-Arveson space and use it to obtain a necessary and sufficient condition for the solvability of the corresponding Nevanlinna-Pick interpolation problem. The vector-valued version and a few examples are also discussed in this section. Finally, Section \ref{invariant subspace} is devoted to the characterization of the shift-invariant subspaces of the nc Drury-Arveson space.

\section{Preliminaries}\label{Sec: pre}
In this section, we recall some preliminary notions and results that will be used subsequently. We first introduce the notation. Throughout the paper, $d\in\N$ is fixed.
For a vector space $\cV$ and $n\in\bN$, we denote $M_n(\cV)$ by the set of all $n\times n$ matrices over $\cV$. Let
\[
M_n^d(\cV):=\{X=(X_1,\ldots,X_d):X_i\in M_n(\cV),\,\forall\, i=1,\ldots,d\}
\]
be the set of all $d$-tuples of $n\times n$ matrices over $\cV$.
Define  
\[
\bM^d(\cV):=\bigsqcup_{n=1}^{\infty} M_n^d(\cV)
\]
which is the disjoint union over $n\in\bN$, of all $d$-tuples of $n\times n$ matrices over $\cV$. \textit{When $\cV=\bC$, we simply write $\bM^d$ instead of $\bM^d(\bC)$}. For a set $\Omega\subseteq\bM^d$, we denote $\Omega(n)=\Omega\cap M_n^d(\C)$, which is the slice of $\Omega$ at the $n$-th level. 

\subsection{Noncommutative sets and noncommutative functions}\label{Sub: nc function} A set $\Omega\subseteq\bM^d$ is said to be an \textit{nc set} if it is closed under direct sums, that is, for $X\in\Omega(n)$ and $Y\in\Omega(m)$
\[
X\oplus Y=\begin{bmatrix}
    X & 0\\
    0 & Y
\end{bmatrix}\in\Omega(n+m).
\]
Let $\cV$ be a vector space and $\Omega\subseteq \bM^d$ be an nc set. A function $f:\Omega\to\bM^1(\cV)$ is said to be an \textit{nc function} (with values in $\cV$) if 
\begin{itemize}
    \item[(i)] $f$ is graded: $X\in\Omega(n)$ implies $f(X)\in M_n(\cV)$.
    \item[(ii)] $f$ respects direct sums: $f(X\oplus Y)=f(X)\oplus f(Y)$.
    \item[(iii)] $f$ respects similarities: if $X\in\Omega(n)$ and $S\in M_n(\C)$ is invertible and if $S^{-1}XS\in\Omega(n)$ then $f(S^{-1}XS)=S^{-1}f(X)S$.
\end{itemize}
It turns out that the last two conditions are equivalent to a single one: \textit{$f$ respects intertwinings: $XS= SY$ implies $f(X)S= Sf(Y)$ for $X\in\Omega(n)$, $Y\in\Omega(m)$ and $S\in M_{n\times m}(\C)$}. This condition originates in the pioneering work of Taylor (\cite{T73}). One canonical example of nc function is that any polynomial $p$ over $\bC$ in $d$-noncommutative variables is an nc function on the nc set $\bM^d$.

There is a natural norm topology on $\bM^d$, and with respect to that topology we define the following open set, namely, the open unit nc row ball
\[
\fB_d:=\Big\{X=(X_1,\ldots,X_d)\in\bM^d:\|X\|^2=\Big\|\sum_{j=1}^dX_jX_j^*\Big\| <1\Big\}.
\] 
The topology above is the disjoint union topology, and $\fB_d$ is open in this topology as $\fB_d\cap M_n^d(\C)$ is open in $M_n^d(\C)$ for every $n$. Another topology one can give on $\bM^d$ is the free topology, which is defined as follows: let $\delta=\big(p_{ij}\big)_{l\times m}$ be a matrix of free polynomials in $d$-variables. We define 
\[
G_{\delta}=\{X\in\bM^d:\|\delta(X)\|<1\}.
\]
The free topology on $\bM^d$ is the topology that has a basis consisting of all such $G_\delta$, where $\delta$ is a matrix of free polynomials. One can verify that $G_{\delta_1}\cap G_{\delta_2}=G_{\delta_1\oplus\delta_2}$, which ensures that set of all $G_{\delta}$ indeed forms a basis for a topology. The free topology is commonly used in the noncommutative function theory.  

An nc set $\Omega\subseteq \bM^d$ is said to be an \textit{nc domain} if it is open in the free topology and if every $\Omega(n)$ is connected. The basic open sets $G_{\delta}$ are some examples of nc domain.

In the following we write down $G_{\delta}$ in some particular cases:
\begin{example}
\begin{enumerate}
    \item If $\delta(x)=[x_1\cdots x_d]$ where $x_1,\ldots,x_d$ are noncommutative variables and $x=(x_1,\ldots,x_d)$, then 
    \begin{align*}
        G_{\delta}&=\big\{X=(X_1,\ldots,X_d)\in\bM^d:\|[X_1\cdots X_d]\| <1\big\}\\
        &=\Big\{X=(X_1,\ldots,X_d)\in\bM^d:\Big\|\sum_{j=1}^dX_jX_j^*\Big\| <1\Big\},
    \end{align*}
which is nothing but the open unit nc row ball $\fB_d$.

\item If $\delta(x)=\text{diag}\{x_1,\ldots,x_d\}$, then 
\begin{align*}
    G_{\delta}&=\big\{X=(X_1,\ldots,X_d)\in\bM^d:\|\text{diag}\{X_1,\ldots,X_d\}\| <1\big\}\\
        &=\big\{X=(X_1,\ldots,X_d)\in\bM^d:\|X_i\|^2=\|X_iX_i^*\|<1,\,\forall\,i=1,\ldots,d \big\},
\end{align*}
which is an open unit nc polydisc.
\end{enumerate}
\end{example}

A \textit{free holomorphic function} $\varphi$ on an open set $\Omega\subseteq \bM^d$ is an nc function that is locally bounded in the free topology, i.e., for every $X\in\Omega$  there is a basic free open set $G_{\delta}\subseteq\Omega$ such that $X\in G_{\delta}$ and $\varphi|_{G_{\delta}}$ is a bounded nc function. We denote the set of bounded free holomorphic functionon an open set $\Omega\subseteq\bM^d$ by $H^{\infty}(\Omega)$ and the closed unit ball of $H^{\infty}(\Omega)$ is denoted by $H^{\infty}_1(\Omega)$, that is,
\[
H^{\infty}_1(\Omega):=\big\{\Phi\in H^{\infty}(\Omega): \sup_{Z\in \Omega}\|\Phi(Z)\|\leq 1 \big\}.
\]

\subsection{Noncommutative kernels and noncommutative reproducing kernel Hilbert spaces}\label{Sub: nc RKHS} Let $\Omega\subseteq\bM^d$ be an nc set, and $\fA$ and $\fC$ be $C^*$-algebras. A completely positive nc (abbreviated as \textit{cp}-nc) kernel (with values in $\cB(\fA,\fC)$) on $\Omega$ is a function
\[
K:\Omega\times\Omega\to\bigsqcup_{m,n=1}^{\infty} \cB\big(M_{m\times n}(\fA), M_{m\times n}(\fC)\big)
\]
such that 
\begin{enumerate}
    \item $K$ is graded: $Z\in\Omega(m)$ and $W\in\Omega(n)$ implies $K(Z,W)\in \cB\big(M_{m\times n}(\fA),M_{m\times n}(\fC)\big)$.

    \item $K$ respects direct sums: For $Z\in\Omega(n)$, $Z'\in\Omega(n')$, $W\in\Omega(m)$ and $W'\in\Omega(m')$ such that $\begin{bmatrix}
        Z & 0\\
        0 & Z'
    \end{bmatrix}\in\Omega(n+n')$ and $\begin{bmatrix}
        W & 0\\
        0 & W'
    \end{bmatrix}\in\Omega(m+m')$ and $P=\begin{bmatrix}
        P_{11} & P_{12}\\
        P_{21} & P_{22}
    \end{bmatrix}\\\in M_{(n+m)\times (n'+m')}(\fA)$, we have
    \[
    K\bigg(\begin{bmatrix}
        Z & 0\\
        0 & Z'
    \end{bmatrix},\begin{bmatrix}
        W & 0\\
        0 & W'
    \end{bmatrix}\bigg)\bigg(\begin{bmatrix}
        P_{11} & P_{12}\\
        P_{21} & P_{22}
    \end{bmatrix}\bigg)=\begin{bmatrix}
        K(Z,W)(P_{11}) & K(Z,W')(P_{12})\\
        K(Z',W)(P_{21}) & K(Z',W')(P_{22})
    \end{bmatrix}.
    \]
    \item $K$ respects similarities: For $Z,Z'\in\Omega(n)$, $W,W'\in\Omega(m)$, $A\in M_n(\C)$ invertible and $B\in M_m(\C)$ invertible if $Z'=AZA^{-1}$, and $W'=BWB^{-1}$ then for all $P\in M_{n\times m}(\fA)$
    \[
    K(Z',W')(P)=AK(Z,W)(A^{-1}PB^{-1*})B^*.
    \]
    \item $K(Z,Z)$ is a \textit{cp} map for all $Z\in\Omega$. 
\end{enumerate} 

In the definition of \textit{cp}-nc kernel, if we take $\fC=\clb(\cE)$, for some Hilbert space $\cE$ then we know the following result, which is a part of \cite[Theorem 3.1]{BMV16}.
\begin{theorem}\label{rkhs}
    Suppose that $\Omega\subseteq \bM^d$ is an nc set, $\cE$ is a Hilbert space, $\fA$ is a $C^*$-algebra and 
    \[
    K:\Omega\times\Omega\to\bigsqcup_{m,n=1}^{\infty} \cB\big(M_{m\times n}(\fA), M_{m\times n}(\cB(\cE))\big) 
    \]
    is a \textit{cp}-nc kernel. Then there is a Hilbert space $H(K)$ whose elements are nc functions 
    \[
    f:\Omega\to\bigsqcup_{n=1}^{\infty} M_n\big(\cB(\fA, \cE)\big)\cong\bigsqcup_{n=1}^{\infty} \cB\big(\fA^n, \cE^n\big)
    \]
    such that
    \begin{enumerate}
        \item For each $W\in\Omega(m)$, $v\in\fA^{m}$ and $y\in\cE^m$, the function
        \[
        K_{W,v,y}:\Omega(n)\to \bigsqcup_{n=1}^{\infty} \cB\big(\fA^n, \cE^n\big),
        \]
     called the kernel function, defined by
        \[
        K_{W,v,y}(Z)u=K(Z,W)(uv^*)y\quad (Z\in\Omega(n),\, u\in\fA^n),
        \]
        belongs to $H(K)$.

        \vspace{0.1cm}

        \item $H(K)=\overline{\text{span}}\{K_{W,v,y}: W\in\Omega(n), y\in\cE^n, v\in\fA^n, n\in \bN\}$.

\vspace{0.1cm}

        \item The kernel function $K_{W,v,y}$, as defined above, have the reproducing property: for $f\in H(K)$, $W\in\Omega(m)$ and $v\in\fA^m$,
        \[
        \big\langle f, K_{W,v,y}\big\rangle_{H(K)}=\big\langle f(W)v,y\big\rangle_{\cE^m}.
        \]
    \end{enumerate}
\end{theorem}
 
We call the Hilbert space $H(K)$ in Theorem \ref{rkhs} as \textit{nc reproducing kernel Hilbert space} corresponding to $K$. 

\subsection{Amplification of noncommutative kernels}
Let $\fA$ be a $C^*$-algebra and $\cE$, $\cF$ be Hilbert spaces. Suppose that $K$ is a \textit{cp}-nc kernel on a set $\Omega$ with values in $\clb(\fA,\clb(\cE))$. Consider the kernel  
\[
\tilde{K}:\Omega \times \Omega \to  \bigsqcup_{m, n=1}^\infty \clb \big(M_{m\times n}(\fA), M_{m\times n}(\clb(\cle \otimes \clf))\big) 
\]
defined by 
\[\tilde{K}=K \otimes I_\clf.\]
More precisely, for $Z\in\Omega(m)$, $W\in \Omega(n)$ and $P\in M_{m\times n}(\fA)$, for all $m,n\in\bN$
\[
\tilde{K}(Z,W)(P)=K(Z,W)(P)\otimes I_{\clf}:(\cle\otimes \clf)^n \cong \cle^n \otimes \clf\to (\cle\otimes \clf)^m\cong \cle^m \otimes \clf
\]
such that 
 \[
 [K(Z,W)(P)](\eta\otimes \zeta)=[K(Z,W)(P)]\eta\otimes \zeta\quad(\eta\in\cle^n, \zeta\in\clf).
 \]
 We call the kernel $\tilde{K}$ as the \textit{amplification of the kernel $K$ by $\clf$}. The corresponding nc reproducing kernel Hilbert space $H(\tilde{K})$ is nothing but $H(K)\otimes\cF$ which is the space of $\clb(\cE\otimes\cF)$ valued nc functions on $\Omega$, that is, $H(\tilde{K})$ is the collection of nc functions 
\[
    f:\Omega\to\bigsqcup_{n=1}^{\infty} M_n\big(\cB(\fA, \cE\otimes\cF)\big)\cong\bigsqcup_{n=1}^{\infty} \cB\big(\fA^n, \cE^n\otimes\cF\big).
\]
The Hilbert space $H(\tilde{K})$ is generated by the set of kernel functions 
\[
\big\{\tilde{K}_{W,v,\eta\otimes\zeta}=K_{W,v,\eta}\otimes \zeta: W\in\Omega(n), \eta\in\cE^n, v\in\fA^n,\zeta\in\clf, n\in \bN\big\}.
\]
By Theorem \ref{rkhs}, one can view the above kernel functions in the following explicit way:
\begin{equation}\label{kernel formula}
\tilde{K}_{W, v, \eta\otimes \zeta}(Z)u=\tilde{K}(Z,W)(uv^*) (\eta\otimes \zeta)= (K(Z,W)(uv^*) \eta) \otimes \zeta,
\end{equation}
and the reproducing property of the kernel $\tilde{K}$ is given by the following formula: 
\begin{equation}\label{reproducing property}
\langle f(W)(v), \eta\otimes \zeta\rangle_{\cle^n\otimes \clf}=\langle f, \tilde{K}_{W, v, \eta\otimes \zeta}\rangle_{\clh(\tilde{K})}=\langle f, K_{W, v, \eta} \otimes \zeta\rangle_{\clh(\tilde{K})},
\end{equation}
for each $W\in \Omega(n)$, $v\in \fA^n$, $\eta\otimes \zeta\in \cle^n\otimes \clf$, $Z\in \Omega(m)$ and $u\in \fA^m$. Also, we have the following inner product formula between two different kernel functions:
\begin{align*}
\langle \tilde{K}_{W, v, \eta\otimes \zeta} , \tilde{K}_{Z, u, \gamma\otimes \xi} \rangle_{\clh(\tilde{K})}&=\langle \tilde{K}_{W, v, \eta\otimes \zeta}(Z)u, \gamma\otimes \xi\rangle_{\cle^m \otimes \clf} \nonumber\\
&=\langle \tilde{K}(Z,W)(uv^*)(\eta\otimes \zeta),\gamma\otimes \xi \rangle_{\cle^m \otimes \clf}\nonumber\\
&=\langle K(Z,W)(uv^*)\eta, \gamma \rangle_{\cle^m} \langle \zeta, \xi \rangle_\clf,
\end{align*}

for all $W\in \Omega(n)$, $v\in \fA^n$, $\eta\otimes \zeta\in \cle^n\otimes \clf$, $Z\in \Omega(m)$ and $u\in \fA^m$.

\section{Noncommutative Drury-Arveson space and multipliers}\label{vector-valued-Drury}
In this section, we study noncommutative Drury-Arveson spaces and their multiplier algebras. In addition to the scalar-valued case, we consider the general vector-valued setting, which is of independent interest and is also used in several results in later sections. For instance, Theorem \ref{dilation} employs the notion of a vector-valued noncommutative Drury-Arveson space, and the proof of Theorem \ref{inv-sub-left-shift} relies on the vector-valued intertwining result established in Theorem \ref{intertwiner of left shifts}. Here, we analyze the left and right multipliers separately, and we will see that they are linked through some commutation relations, developed and used in later sections. 

 \subsection{Noncommutative Drury-Arveson space:}\label{Sub: nc Drury} In order to define the noncommutative Drury-Arveson space, first we consider the the nc Szeg\H o kernel on $\fB_d$, we denote it by $\clk$, where
    \[
   \clk:\fB_d\times\fB_d\to\bigsqcup_{m,n=1}^{\infty} \cB\big(M_{m\times n}(\bC), M_{m\times n}(\bC)\big) 
    \]
    defined by
    \[
   \clk(Z,W)(P)=\sum_{\alpha\in \bW_d}Z^\alpha P(W^\alpha)^*,
    \]
     for $Z\in\fB_d(m)$, $W\in \fB_d(n)$ and $P\in M_{m\times n}(\bC)$. Here $\bW_d$ is the free monoid on $d$ generators, denoted here by the first $d$ natural numbers $\{1,...,d\}$. An element $\alpha\in\bW_d$ is a word the form $\alpha=\alpha_1\cdots \alpha_p$, where each $\alpha_i$ ($i=1,\ldots,p$) belongs to $\{1,\ldots,d\}$. For two elements $\alpha=\alpha_1\cdots \alpha_p, \beta=\beta_1\cdots \beta_t\in \bW_d$, $\alpha\cdot \beta$ is defined through the juxtaposition
     \[
     \alpha\cdot \beta:=\alpha_1\cdots \alpha_p\beta_1\cdots \beta_t.
     \]
    And, for an element $Z\in \fB_d$ and for $\alpha=\alpha_1\cdots \alpha_p\in \bW_d$,
    \[
    Z^{\alpha}:=Z_{\alpha_1}\cdots Z_{\alpha_p},
    \]
    with the convention that $Z^{\emptyset}=1$, where $\emptyset$ denotes the empty word.
     The nc reproducing kernel Hilbert space corresponding to $\cK$ is called the \textit{nc Drury-Arveson space} and denoted as $\cH^2_d$.  The space $\cH^2_d$ is generated by the kernel functions $\cK_{W,v,y}$ for $W\in\fB_d(m)$ and $v,y\in\bC^n$, where we can write the kernel function in explicit form as follows:
     \begin{equation}\label{Drury explicit}
    \clk_{W,v,y}(Z)u=\clk(Z,W)(uv^*)y=\sum_{\alpha\in\bW_d}Z^\alpha uv^*(W^\alpha)^*y=\sum_{\alpha\in\bW_d}\langle y, W^\alpha v\rangle Z^\alpha u,
     \end{equation}
for $Z\in\fB_d(n)$ and $u\in\bC^n$. 
It is known from  \cite[Corollary 3.7]{SSS} and \cite[Theorem 5.5]{BMV18}, that $\text{Mult}_L(\clh^2_d)\cong H^\infty(\fB_d)$, and $\|\Phi\|_\infty=\|M_\Phi\|$.

We also consider the vector-valued nc Drury-Arveson space $H_\clf(\clk)$, for some Hilbert space $\clf$, corresponding to the kernel
\[
\Tilde{\clk}:\fB_d \times \fB_d \to  \bigsqcup_{m, n=1}^\infty \clb \big(M_{m\times n}(\C), M_{m\times n}(\clb(\C \otimes \cle))\big) 
\]
defined by 
\[
\Tilde{\clk}=\clk \otimes I_\clf.
\]
Using the terminology of the previous section, $\Tilde{\clk}$ is the amplification of the kernel $\clk$ by $\clf$. Combining the equations (\ref{kernel formula}) and (\ref{Drury explicit}) we have the following explicit form of the vector-valued Drury-Arveson kernel functions:
\[
\Tilde{\clk}_{W, v,\eta\otimes \zeta}(Z)u=\Big (\sum_{\alpha\in \mathbb{W}_d} \la \eta, W^\alpha v\ra Z^\alpha u \Big )\otimes \zeta,
\]
for each $W\in\fB_d(n)$, $v\in \C^n$, $\eta \otimes \zeta \in \C^n\otimes \clf$, $Z\in\fB_d(m)$, $u\in \C^m$.

\subsection{Left multipliers}
Let $\cle$ and $\clf$ be two auxiliary Hilbert spaces. An nc function 
$\Phi$ on $\fB_d$ with values $\clb(\cle, \clf)$, that is,
\[
\Phi:\fB_d \to \bigsqcup_{n=1}^\infty M_n(\clb(\cle, \clf))\cong\bigsqcup_{n=1}^{\infty} B\big(\cle^n, \clf^n\big)
\]
is said to be a \textit{left multiplier} from $H_\cle(\clk)$ to $H_\clf(\clk)$ if 
\[\Phi f\in H_\clf(\clk)\quad(f\in H_\cle(\clk)),
\]
where $(\Phi f)(W)=\Phi(W) f(W)$ for $W\in \fB_d$. 

The multiplier $\Phi$ always associates an operator
\[M_\Phi :H_\cle(\clk)\to H_\clf(\clk)\]
defined by 
\[
(M_\Phi f)(W)=\Phi(W) f(W) \quad (W\in \fB_d).
\]
Using closed graph theorem, we prove that $M_\Phi$ is a bounded operator via the following line of argument. First observe that, for a sequence $\{f_k\}_{k\geq 1}$ and a function $f$ in $H_\cle(\clk)$
\[
|\langle (f_k(W) - f(W) )v, \eta \rangle|=|\langle (f_k- f ), K_{W, v, \eta} \rangle|\leq \|f_k-f\|\|K_{W, v,\eta}\|,
\]
for each $W\in \Omega_m$, $v\in \C^m$, and $\eta\in \cle^m$. If $\{f_k\}_{k\geq 1}$ converges to $f$ in $H_\cle(\clk)$, then for each $W\in \Omega_m$, $v\in \C^m$, $\{f_k(W)v\}_{k\geq 1}$ converges to $f(W)v$ in  $\cle^m$. Let $\{M_{\Phi}f_k\}_{k\geq 1}$ is a sequence in $H_\clf(\clk)$ and it converges to $g$ in $H_\clf(\clk)$. Therefore, by the above, $\{\Phi(W)f_k(W)v\}_{k\geq 1}$  converges to $g(W)v$  in $\clf^m$. So,
\[
g(W)v =\lim_{k\to \infty} \Phi(W)f_k(W) v = \Phi(W)f(W) v,
\]
for all $v\in \C^m$. This shows that $M_\Phi$ is bounded.

We write a part of the above argument separately as a remark for future references.
\begin{remark}\label{Remark_f_n}
If $\{f_k\}_{k\geq 1}$ is a sequence in $H_\cle(\clk)$ and that converges to $f$ in $H_\cle(\clk)$, then for each $W\in \fB_d(m)$ and $v\in \C^m$, $f_k(W)v$ converges to $f(W)v$ in $\cle^m$.
\end{remark}

The set of all left multipliers from $H_\cle(\clk)$ to $H_\clf(\clk)$ is denoted by  $\text{Mult}_L(H_\cle(\clk), H_\clf(\clk))$, that is, 
\begin{align*}
\text{Mult}_L&(H_\cle(\clk), H_\clf(\clk))\\
&:=\Big \{\Phi: \fB_d \to \bigsqcup_{n=1}^\infty M_n(\clb(\cle, \clf)): \Phi \,\text{is nc}\,\, \& \,\,\Phi f\in H_\clf(\clk),  \forall f\in H_\cle(\clk) \Big \}.
\end{align*}
The following two results are useful for our purpose.
\begin{proposition}\label{action_adj_mul}
Let $\Phi\in \text{Mult}_L(H_\cle(\clk), H_\clf(\clk))$. Then for each $W\in \fB_d$, $v\in \C^n$, and $y\otimes \zeta\in \C^n\otimes \clf$,
\[
M^*_{\Phi} \big(\clk_{W, v, y}\otimes \zeta\big )= \Tilde{\clk}_{W, v, \Phi^*(W) (y\otimes\eta)},
\]
where $\Tilde{\clk}=\clk\otimes I_{\cle}$.
\end{proposition}
\begin{proof}
For $f\in H_\cle(\clk)$,
\begin{align*}
\langle f, M^*_{\Phi} \big(\clk_{W, v, y}\otimes \zeta\big)  \rangle=\langle M_\Phi f,  \clk_{W, v, y}\otimes \zeta \rangle &=\langle \Phi (W)f(W)v, y \otimes \zeta \rangle\\
&=\langle f(W)v, \Phi (W)^* (y\otimes\zeta) \rangle\\
&= \langle f,  \Tilde{\clk}_{W, v,\Phi(W)^* (y\otimes \zeta)} \rangle.
\end{align*}
This completes the proof.
\end{proof}

\begin{proposition}\label{inequ_elem}
Let $H_\cle(\clk)$ be an nc Drury-Areveson ($\cle$-valued) space on $\fB_d$. Then 
\[\|f(Z)\|_{\clb(\C^m, \C^m\otimes\cle )}\leq \|f\|_{H_\cle(\clk)} \|\clk(Z,Z)\|^{1/2}_{\clb (M_{m}(\C))},\]
for all $Z\in \fB_d(m)$ and  $f\in H_\cle(\clk)$.
\end{proposition}
\begin{proof}
For $Z\in\fB_d(m)$, $u\in \C^m$, and $\gamma \otimes \xi \in \C^m\otimes \cle$ with $\|u\|=1$ and $\|\gamma\otimes \xi\|=1$,
\begin{align*}
|\la f(Z)u, \gamma\otimes \xi  \ra|_{\C^m \otimes \cle}^2=|\la f, \Tilde{\clk}_{Z, u, \gamma\otimes \xi} \ra|_{H_\cle(\clk)}^2 & \leq \|f\|_{H_\cle(\clk)}^2 \langle K_{Z, u, \gamma\otimes \xi} , K_{Z, u, \gamma\otimes \xi} \rangle_{H_\cle(\clk)}\\
&= \|f\|_{H_\cle(\clk)}^2 \langle K_{Z, u, \gamma\otimes \xi}(Z)u, \gamma\otimes \xi\rangle_{\C^m \otimes \cle}\\ 
&=\|f\|_{H_\cle(\clk)}^2\langle K(Z,Z)(uu^*)(\gamma\otimes \xi),\gamma\otimes \xi\rangle_{\C^m \otimes \cle}\\
&=\|f\|_{H_\cle(\clk)}^2\langle \clk(Z,Z)(uu^*)\gamma, \gamma \rangle_{\C^m} \langle \xi, \xi \rangle_\cle\\
&\leq \|f\|_{H_\cle(\clk)}^2 \|\clk(Z,Z)(u u^*)\|_{M_{m}(\C)}\\
&\leq \|f\|_{H_\cle(\clk)}^2 \|\clk(Z,Z)\|_{\clb (M_{m}(\C))}.
\end{align*}
The second last inequality follows by taking the supremum over $\gamma\otimes \xi$. Therefore, \[|\la f(Z)u, f(Z)u \ra|_{\C^m \otimes \cle}^2\leq \|f\|_{H_\cle(\clk)}^2 \|\clk(Z,Z)\|_{\clb (M_{m}(\C))}.\]
By taking the supremum over $u$, we have the desired inequality.
\end{proof}

We call the tuple of operators $M_z:=(M_{z_1} \otimes I_{\cle}, \ldots, M_{z_d}\otimes I_\cle)$ the tuple left shifts on $H_\cle(\clk)$, where the components are multiplication by coordinate functions, that is,  
\[
\big[\big(M_{z_i} \otimes I_{\cle} \big)f\big](W)= (W_i\otimes I_{\cle}) f(W)=\sum_{\alpha} W^{i\cdot \alpha}  f_\alpha,
\]
for $f(Z)=\sum_{\alpha\in \mathbb{W}_d} Z^\alpha f_\alpha\in H_\cle(\clk)$.

\subsection{Right multipliers and right product}
An nc function 
$\Psi$ on $\fB_d$ with values in $\clb(\cle, \clf)$, that is,
\[
\Psi:\fB_d \to \bigsqcup_{n=1}^\infty M_n(\clb(\cle, \clf))\cong\bigsqcup_{n=1}^{\infty} B\big(\cle^n, \clf^n\big)
\]
is said to be a \textit{right multiplier} from $H_\cle(\clk)$ to $H_\clf(\clk)$ if the \textit{right product} $\Psi \bullet_R f\in H_\cle(\clk)$, for all $f\in H_\cle(\clk)$, where
\[
(\Psi \bullet_R f)(W)=\Psi(W) \bullet_R f(W)\quad(W\in \fB_d),
\] 
and $\Psi(W) \bullet_R f(W)$ is defined in the following way: first, let  $\Psi(Z)=\sum_{\alpha\in \mathbb{W}_d} Z^\alpha \otimes \psi_\alpha$, where $\psi_\alpha\in \clb(\cle, \clf)$. Also, let $f\in H_\cle(\clk)=H^2(\fB_d)\otimes \cle$ and then write it as $f(Z)=\sum_{\alpha\in \mathbb{W}_d} Z^\alpha\otimes f_\alpha$, where $f_\alpha\in \cle$. the right product $(M^R_\Psi f)$ is given by
\[(M^R_\Psi f)(Z)=\Psi(Z) \bullet_R f(Z):= \sum_{\alpha, \beta} Z^{\beta \cdot \alpha}\otimes \psi_\alpha f_\beta.\]
Note that, in the scalar-valued setting, the right product reduces to $\Psi (Z) \bullet_R f(Z)= f(Z) \Psi(Z)$. More generally, an nc function $\Psi$ is said to be a \textit{right multiplier} from $H_\cle(\clk)$ to $H_\clf(\clk)$ if 
\begin{equation}\label{right_adjoint_action}
(M^R_\Psi)^* \big (\clk_{W, v, \eta}\otimes \zeta\big )=(M^R_\Psi)^* \big (K_{W, v, \eta\otimes \zeta}\big ):= \clk_{W, \Psi(W)v, \eta}\otimes \zeta=K_{W, \Psi(W)v, \eta\otimes \zeta}, 
\end{equation}
for $W\in\fB_d(n)$, $v\in \C^n$, $\eta \otimes \zeta \in \C^n\otimes \clf$. Consequently, for $f\in H_\cle(\clk)$, we define 
\[
\langle \Psi(W) \bullet_R f(W) v, \eta \otimes \zeta\rangle_{\clf^n}:=\langle f, K_{W, \Psi(W)v, \eta\otimes \zeta}\rangle_{H_\cle(\clk)}.
\]
For more details see \cite{JM}. The set of all right multipliers is denoted by $\text{Mult}_R(H_\cle(\clk), H_\clf(\clk))$, that is, 
\begin{align*}
\text{Mult}_R&(H_\cle(\clk), H_\clf(\clk))\\
&:=\Big \{\Psi: \fB_d \to \bigsqcup_{n=1}^\infty M_n(\clb(\clf, \cle)):\Psi \,\text{is nc}\,\, \& \,\,\Psi \bullet_R f\in H_\clf(\clk),  \forall f\in H_\cle(\clk) \Big \}.
\end{align*}
Similar to the case of lift multipliers, for $\Psi\in \text{Mult}_R(H_\cle(\clk), H_\clf(\clk))$, $M^R_\Psi $ is a bounded operator from $H_\cle(\clk)$ to $H_\clf(\clk)$.

The tuple of right shifts on $H_\cle(\clk)$, denoted by $\Tilde{M}_z:=(\Tilde{M}_{z_1} \otimes I_{\cle}, \ldots ,\Tilde{M}_{z_d}\otimes I_{\cle})$, defined via the identification 
\begin{align*}
\big[\big(\Tilde{M}_{z_i} \otimes I_{\cle} \big)f\big](W)= (W_i\otimes I_{\cle})\bullet_R f(W)=\sum_{\alpha} W^{\alpha \cdot i} \otimes f_\alpha,
\end{align*}
for $f(z)=\sum_{\alpha\in \mathbb{W}_d} Z^\alpha \otimes f_\alpha\in H^2(\fB_d)\otimes \cle \cong H_\cle(\clk)$. 

One can verify that the following result follows from the definition of right shifts given above. However, we give a direct proof it using the definition of right product by the coordinate functions.
\begin{proposition}\label{action_adj_right_shift}
For each $W\in \fB_d$, $v\in \C^n$, and $y\otimes \zeta\in \C^n\otimes \clf$,
\[
(\Tilde{M}_{z_i} \otimes I_{\cle})^*\big(\clk_{W, v, y}\otimes \zeta\big )= \Tilde{\clk}_{W, W_iv, y\otimes\zeta},
\]
where $\Tilde{\clk}=\clk\otimes I_{\cle}$.
\end{proposition}
\begin{proof}
Let $f\in H_\cle(\clk)$. Suppose $f(Z)=\sum_{\alpha\in \mathbb{W}_d} Z^\alpha \otimes f_\alpha$. Then,
\begingroup
 \allowdisplaybreaks
\begin{align*}
\langle f, (\Tilde{M}_{z_i} \otimes I_{\cle})^*\big(\clk_{W, v, y}\otimes \zeta\big)  \rangle&=\langle (\Tilde{M}_{z_i} \otimes I_{\cle}) f,  \clk_{W, v, y}\otimes \zeta \rangle\\
&=\Big\langle\sum_{\alpha} Z^{\alpha.i}\otimes  f_\alpha,\Big (\sum_{\alpha} \la y, W^\alpha v\ra Z^\alpha \Big )\otimes \zeta \Big\rangle\\
&=\sum_{\alpha}\la y, W^\alpha W_i v \ra\la f_{\alpha},\zeta \ra\\
&=\Big\langle\sum_{\alpha} Z^{\alpha}\otimes  f_\alpha,\Big (\sum_{\alpha} \la y, W^\alpha W_i v\ra Z^\alpha \Big )\otimes \zeta \Big\rangle\\
&= \langle f,  \Tilde{\clk}_{W, W_iv,y\otimes \zeta} \rangle.
\end{align*}
\endgroup
This completes the proof.
\end{proof}

We end this subsection proving that the tuple of right shifts is a row isometry. We will use this property to prove a commutant lifting theorem in Section \ref{lifting and interpolation}.
Recall that an operator tuple $T=(T_1,\ldots,T_d)$ is said to be a row contraction on a Hilbert space $\clh$ if 
\[
\sum_{i=1}^d T_iT_i^*\leq I.
\]
A row contraction $T=(T_1,\ldots,T_d)$ is said to be a row isometry if
\[
T_i^*T_i=I\,\,(\text{for all}\,\, i=1,\ldots,d).
\]
\begin{proposition}\label{right_shift_row}
The tuple of right shifts $\Tilde{M}_z=(\tilde{M}_{z_1},\ldots,\Tilde{M}_{z_d})$ on $\clh^2_d$ is a row isometry. 
\end{proposition}
\begin{proof}
First we claim that 
\[
\Big(I-\sum_{i=1}^d\Tilde{M}_{z_i}\Tilde{M}_{z_i}^*\Big)f=P_{\text{const.}}f\quad\big(f\in\clh^2_d\big),
\]
where $P_{\text{const.}}$ is the orthogonal projection of $\clh^2_d$ onto the set of all constant functions in $\clh^2_d$. To justify the claim, we aim to prove for all $f\in\clh^2_d$
\begin{equation}\label{proj_onto_const}
\Big(I_n-\sum_{i=1}^d\Tilde{M}_{z_i}\Tilde{M}_{z_i}^*\Big)f(Z)=P_{\text{const.}}f(0_n)\quad\big( Z\in\fB_d(n)\big).
\end{equation}
Indeed,
\begingroup
 \allowdisplaybreaks
\begin{align*}
\Big\langle \Big(I_n-&\sum_{i=1}^d\Tilde{M}_{z_i}\Tilde{M}_{z_i}^*\Big)\clk_{Z,u,x},\clk_{W,v,y}\Big\rangle\\
&=\la\clk(W,Z)(vu^*)x,y\ra-\sum_{i=1}^d\la\clk(W,Z)(vu^*)Z_i^*x,W_i^*y\ra\\
&=\Big\langle\Big[\clk(W,Z)(vu^*)-\sum_{i=1}^d \clk(W,Z)(W_ivu^*Z_i^*)\Big]x,y\Big\rangle\\
&=\la[vu^*]x,y\ra\\
&=\la\clk_{0_n,u,x},\clk_{W,v,y}\ra. 
\end{align*}
\endgroup
So,
\[
\Big(I_n-\sum_{i=1}^d \Tilde{M}_{z_i}\Tilde{M}_{z_i}^*\Big)\clk_{Z,u,x}(W)=\clk_{0_n,u,x}(W)=\sum_{\alpha\in\bW_d}\langle x,0_n^{\alpha} u\rangle W^{\alpha}=\clk_{Z,u,x}(0_n).
\]
Therefore, through a limiting argument, the relation (\ref{proj_onto_const}) and consequently the claim is proved. 

From the claim it follows that
\[
\Big(I-\sum_{i=1}^d\Tilde{M}_{z_i}\Tilde{M}_{z_i}^*\Big)=P_{\text{const.}}\geq 0.
\]
Hence, $(\Tilde{M}_{z_1},\ldots,\Tilde{M}_{z_d})$ is a row contraction. 
Now, for $f(Z)=\sum_{\alpha\in\bW_d}Z^{\alpha}f_{\alpha}$
\[
\|\Tilde{M}_{z_i}f\|^2=\Big\langle \sum_{\alpha\in\bW_d}Z^{\alpha.i}f_{\alpha},\sum_{\alpha\in\bW_d}Z^{\alpha.i}f_{\alpha}\Big\rangle=\|f\|^2.
\]
This shows that for each $i=1,\ldots, d$, $\Tilde{M}_{z_i}$ is an isometry. This completes the proof.
\end{proof}

\section{Commutants of the shifts on noncommutative Drury-Arveson space}\label{commutant}
In this section, our aim is to characterize the commutants of the left and right shifts on nc  Drury–Arveson space. Those characterizations will be used in the next section to establish the commutant lifting theorems. We treat the cases of the right and left shifts separately.

In the following, we prove the commutativity of the right shifts and multiplication operators associate to the left multipliers. 
\begin{proposition}\label{Left_mult_and_shift}
Let $\Phi\in \text{Mult}_L(H_\cle(\clk), H_\clf(\clk))$, and let $(\Tilde{M}_{z_1} \otimes I_{\cle}, \ldots \Tilde{M}_{z_d}\otimes I_{\cle})$ and $(\Tilde{M}_{z_1} \otimes I_{\clf}, \ldots \Tilde{M}_{z_d}\otimes I_{\clf})$ be the right shifts acting on $H_\cle(\clk)$ and $H_\clf(\clk)$, respectively. Then
\[
M_\Phi \big (\Tilde{M}_{z_i} \otimes I_{\cle}\big)= \big(\Tilde{M}_{z_i} \otimes I_{\clf}\big) M_\Phi \quad (i=1, \ldots, d).
\]

\end{proposition}

\begin{proof}
Let $f\in H_\cle(\clk)$. For $W\in\fB_d(n)$, $v\in \C^n$, $\eta \otimes \zeta \in \C^n\otimes \clf$,
\begingroup
 \allowdisplaybreaks
\begin{align*}
\big\langle M_\Phi \big (\Tilde{M}_{z_i}  \otimes I_{\cle}\big)f, \clk_{W, v, \eta}\otimes \zeta\big\rangle &=\big \langle \big ( \Tilde{M}_{z_i} \otimes I_{\cle}\big) f, M^*_\Phi \big( \clk_{W,v, \eta}\otimes \zeta\big )\big\rangle \\
&= \big\langle \big ( \Tilde{M}_{z_i} \otimes I_{\cle}\big) f, \Tilde{\clk}_{W,v, \Phi(W)^*(\eta\otimes \zeta)}\big\rangle \quad (\text{by Proposition }\ref{action_adj_mul})\\
&= \big\langle f,  \big ( \Tilde{M}_{z_i} \otimes I_{\cle}\big)^* \Tilde{\clk}_{W,v, \Phi(W)^*(\eta\otimes \zeta)}\big\rangle\\
&= \la f,   \Tilde{\clk}_{W,W_iv, \Phi(W)^*(\eta\otimes \zeta)}\ra\quad (\text{by Proposition }\ref{action_adj_right_shift})\\
&= \la M_{\Phi}f,   \Tilde{\clk}_{W,W_iv, \eta\otimes \zeta}\ra\quad (\text{by Proposition }\ref{action_adj_mul})\\
&= \big\langle \big (\Tilde{M}_{z_i}  \otimes I_{\cle}\big)M_{\Phi}f,  \clk_{W, v, \eta}\otimes \zeta\big\rangle\quad (\text{by Proposition }\ref{action_adj_right_shift}).
\end{align*}
\endgroup
Hence the result follows.
\end{proof}

Similarly, we prove the commutativity of left shifts and multiplication operators corresponding to right multipliers.
\begin{proposition}\label{Right_mult_and_shift}
Let $\Psi\in \text{Mult}_R(H_\cle(\clk), H_\clf(\clk))$, and let $(M_{z_1} \otimes I_{\cle}, \ldots M_{z_d}\otimes I_{\cle})$ and $(M_{z_1} \otimes I_{\clf}, \ldots M_{z_d}\otimes I_{\clf})$ be the left shifts acting on $H_\cle(\clk)$ and $H_\clf(\clk)$, respectively. Then
\[
M_\Psi \big (M_{z_i} \otimes I_{\cle}\big)= \big(M_{z_i} \otimes I_{\clf}\big) M_\Psi \quad (i=1, \ldots, d).
\]
\end{proposition}

\begin{proof}
Let $f\in H_\cle(\clk)$. For $W\in\fB_d(n)$, $v\in \C^n$, $\eta \otimes \zeta \in \C^n\otimes \clf$,
\begingroup
 \allowdisplaybreaks
\begin{align*}
\big \langle M^R_\Psi \big (M_{z_i}  \otimes I_{\cle}\big)f, \clk_{W, v, \eta}\otimes \zeta\big\rangle &= \la \big ( M_{z_i} \otimes I_{\cle}\big) f, M^{R*}_\Psi \big( \clk_{W,v, \eta}\otimes \zeta\big )\big\rangle \\
&= \big \langle \big ( M_{z_i} \otimes I_{\cle}\big) f, \Tilde{\clk}_{W,\Psi(W)v, \eta\otimes \zeta}\big\rangle\quad (\text{by \eqref{right_adjoint_action}})\\
&= \big \langle f,  \big ( M_{z_i} \otimes I_{\cle}\big)^* \Tilde{\clk}_{W,\Psi(W)v, \eta\otimes \zeta}\big\rangle \\
&= \big \langle f,   \Tilde{\clk}_{W,\Psi(W)v, W_i\eta\otimes \zeta}\big\rangle\quad (\text{by Proposition }\ref{action_adj_mul})\\
&= \la M^R_{\Psi}f,   \Tilde{\clk}_{W, v, W_i\eta\otimes \zeta}\ra\quad (\text{by \eqref{right_adjoint_action}})\\
&= \la \big (M_{z_i}  \otimes I_{\cle}\big)M^R_{\Psi}f,  \clk_{W, v, \eta}\otimes \zeta\ra\quad (\text{by Proposition }\ref{action_adj_mul}).
\end{align*}
\endgroup
Hence, the result follows.
\end{proof}

Let us introduce a few more notations. For $\eta\in \cle$, we consider the function 
$$g_\eta:=(I\otimes \eta):\fB_d \to \bigsqcup_{n=1}^\infty M_n(\clb(\C, \C\otimes \cle)),$$ 
where, for $Z\in \fB_d(m)$,
\[g_\eta(Z)= (I_m\otimes \eta)(Z):\C^m \to \C^m \otimes \cle\]
is defined by 
\[g_\eta(Z)x= x\otimes \eta \quad (x\in\C^m).\]
Clearly, for $\eta\in \cle$, $g_\eta\in H_\cle(\clk)$. Moreover, $\|g_\eta\|_{H_\cle(\clk)}=\|\eta\|_\cle$. 

Let $X$ be a bounded operator from $H_\cle(\clk)$ into $H_\clf(\clk)$. We define a function 
\[
\Phi:\fB_d\to \bigsqcup_{n=1}^\infty M_n(\clb(\cle, \clf))\cong \bigsqcup_{n=1}^\infty \clb(\cle^n, \clf^n) \cong \bigsqcup_{n=1}^\infty \clb(\C^n\otimes\cle, \C^n\otimes\clf),
\] 
where, for $Z\in \fB_d(n)$, $\Phi(Z)$ is defined on the elementary tensor elements of $\C^n\otimes\cle$ in the following way:
\begin{align}\label{mult_id}
\Phi(Z)(x\otimes \eta)= [X(g_\eta)](Z)x= [X(1\otimes \eta)](Z) (x) \quad (x\in\C^n, \eta\in\cle).
\end{align}

Then we extend $\Phi(Z)$ to a linear operator on the linear combinations of the elementary tensor elements in $\C^n\otimes \cle$ by defining
\[
\Phi(Z)\Big(\sum_{i=1}^k x_i\otimes \eta_i\Big)=\sum_{i=1}^k X(g_{\eta_i})(Z)x_i= \sum_{i=1}^k X(1\otimes \eta_i)(Z) (x_i),\]
for $x_i\in\C^n$ and $\eta_i\in\cle$ ($i=1,\ldots,k$).

In the next proposition, we prove that the function $\Phi$, defined in \eqref{mult_id}, is an nc function.
\begin{proposition}
The function $\Phi$, defined in \eqref{mult_id}, is an nc function.
\end{proposition}
\begin{proof}
We first show that, for all $Z\in \fB_d(n)$, $\Phi(Z)$, which is defined above, extends to a bounded linear operator from $\C^n \otimes \cle$ into $\C^n \otimes \clf$. To this end, for $x\otimes \eta\in \C^n\otimes \cle$, we calculate,
\begingroup
 \allowdisplaybreaks
\begin{align*}
\|\Phi(Z)(x\otimes \eta)\|_{\C^n\otimes \clf}
&=\|(X g_\eta)(Z)x\|_{\C^n\otimes \clf}\\
&\leq\|(X g_\eta)(Z)\|_{\clb(\C^n, \C^n\otimes \clf)} \|x\|_{\C^n}\\
&\leq \|\clk(Z,Z)\|^{1/2}_{\clb (M_{n}(\C))} \|X g_\eta\|_{H_\clf(\clk)} \|x\|_{\C^n} \quad (\text{by Proposition  \ref{inequ_elem}})\\
&\leq \|\clk(Z,Z)\|^{1/2}_{\clb (M_{n}(\C))} \|X\| \| g_\eta\|_{H_\cle(\clk)} \|x\|_{\C^n} \\
&= \|\clk(Z,Z)\|^{1/2}_{\clb (M_{n}(\C))} \|X\| \| \eta\|_{\cle} \|x\|_{\C^n} \\
&= \|\clk(Z,Z)\|^{1/2}_{\clb (M_{n}(\C))} \|X\| \|x\otimes\eta\|_{\C^n\otimes \cle}. 
\end{align*}
\endgroup
This shows that for $Z\in \fB_d$, $\Phi(Z)$ is bounded on elementary tensors vectors of $\C^n \otimes \cle$. Thus, $\Phi(Z)$ is bounded on an orthonormal basis $\{x_i\otimes \eta_j: i=1,\ldots,n\, \text{and}\, j\in\bN\}$ (where $x_i$, $i=1,\ldots, n$, is an orthonormal basis of $\C^n$ and $\{\eta_j\}_{j\in\bN}$ is an  orthonormal basis of $\cle$). 
Indeed, for $c_1, \ldots, c_k\in \C$,
\begin{align*}
\Big\|\Phi(Z)\Big(\sum_{i=1}^k c_i(x_i\otimes\eta_i)\Big)\Big\|_{\C^n \otimes \clf}^2&=\Big\|\Big(\sum_{i=1}^k \Phi(Z)(c_i x_i\otimes\eta_i)\Big)\Big\|_{\C^n \otimes \clf}^2\\
&\leq \|\clk(Z,Z)\|_{\clb (M_{m}(\C))} \|X\|^2 \sum_{i=1}^k \|c_i (x_i\otimes\eta_i)\|_{\C^n \otimes \cle}^2\\
&= \|\clk(Z,Z)\|_{\clb (M_{m}(\C))} \|X\|^2 \Big\|\sum_{i=1}^k c_i( x_i\otimes\eta_i)\Big\|_{\C^n \otimes \cle}^2 ,
\end{align*}
where the last equality follows from the Pythagorean theorem. Therefore, $\Phi(Z)$ is bounded on $\text{span}\, \{e_i\otimes \eta_j: i=1,\ldots,n\, \text{and}\, j\in\bN\}$. Therefore $\Phi(Z)$ has a bounded extension (by taking limit on Cauchy sequences) on $\C^n\otimes\cle$, denoted again by $\Phi(Z)$, to abuse the notation.
Hence $\Phi$ is a well-defined function. It is clear that $\Phi$  is nc function since $Xg_\eta$ is an nc function.
\end{proof}
The next theorem characterizes the intertwiners of right shifts. 

\begin{theorem}\label{intertwiner of right shifts}
Let $(\Tilde{M}_{z_1} \otimes I_{\cle}, \ldots, \Tilde{M}_{z_d}\otimes I_{\cle})$ and $(\Tilde{M}_{z_1} \otimes I_{\clf}, \ldots, \Tilde{M}_{z_d}\otimes I_{\clf})$ be the right shifts acting on $H_\cle(\clk)$ and $H_\clf(\clk)$, respectively. Suppose $X$ is a bounded operator from $H_\cle(\clk)$ into $H_\clf(\clk)$. Then $X \big (\Tilde{M}_{z_i} \otimes I_{\cle}\big)= \big(\Tilde{M}_{z_i} \otimes I_{\clf}\big) X$ $(i=1, \ldots, d)$ if and only if $X=M_\Phi$ for some $\Phi\in \text{Mult}_L(H_\cle(\clk), H_\clf(\clk))$.
\end{theorem}
\begin{proof}
Let $\Tilde{M}_\cle=(\Tilde{M}_{z_1} \otimes I_{\cle}, \ldots, \Tilde{M}_{z_d}\otimes I_{\cle})$ and $\Tilde{M}_\clf=(\Tilde{M}_{z_1} \otimes I_{\clf}, \ldots, \Tilde{M}_{z_d}\otimes I_{\clf})$. Since $X \big (\Tilde{M}_{z_i} \otimes I_{\cle}\big)= \big(\Tilde{M}_{z_i} \otimes I_{\clf}\big) X$ $(i=1, \ldots, d)$, we have 
\begin{align*}
X \Tilde{M}^\alpha_\cle=\Tilde{M}^\alpha_\clf X\quad (\alpha\in \mathbb{W}_d).
\end{align*}
Now, for $Z\in\fB_d(n)$ and $\eta\in \cle$,
\begin{align*}
(X \Tilde{M}^\alpha_\cle) g_\eta(Z)=(X \Tilde{M}^\alpha_\cle) (I_n\otimes \eta)(Z)=X (Z^{\alpha^\dag} \otimes \eta),
\end{align*}
where for $\alpha=\alpha_1\cdots \alpha_k\in \mathbb{W}_d$, $\alpha^\dag=\alpha_k\cdots \alpha_1$.
On the other hand, for $Z\in\fB_d(n)$, $\eta\in \cle$, and $x\in \C^n$,
\begin{align*}
(\Tilde{M}^\alpha_\cle X) g_\eta(Z)x=(\Tilde{M}^\alpha_\cle) (X g_\eta)(Z)x
&=(\Tilde{M}^\alpha_\cle)\bullet_R (Xg_\eta)(Z)x\\
&=(Xg_\eta)(Z) (Z^{\alpha^{\dag}} x)\\
&=\Phi(Z)(Z^{\alpha^{\dag}} x \otimes \eta)\quad (\text{by the definition of $\Phi$})\\
&=\Phi(Z)(Z^{\alpha^{\dag}}  \otimes \eta)(x).
\end{align*}
Hence we have 
\begin{align}\label{qual_X}
X(Z^{\alpha^{\dag}} \otimes\eta)=\Phi(Z)(Z^{\alpha^{\dag}}  \otimes \eta).
\end{align}
 It is easy to see that $\cls=\text{span}\, \{Z^{\alpha^{\dag}}\otimes\eta : Z\in\fB_d, \eta\in \cle, \alpha\in \mathbb{W}_d\}$ is dense in $H_\cle(\clk)$. Therefore,
\begin{align}\label{qual_X_S}
Xf(Z)=\Phi(Z)f(Z) \quad (f\in \cls).
\end{align}
Let $\{f_k\}_{k\geq 1}$ be a sequence in $\cls$ and that converges to $f$ in $H_\cle(\clk)$. Then by Remark \ref{Remark_f_n}, for each $Z\in \fB_d(n)$ and $x\in \C^n$, $f_k(Z)x$ converges  $f(Z)x$ in  $\cle^n$. Clearly, $X f_k$ converges to $X f$, and thus, for each $Z\in \fB_d(n)$ and $x\in \C^n$, $Xf_k(Z)x$ converges  $Xf(Z)x$ in  $\cle^n$. 
Now 
\begin{align*}
Xf(Z)x=\lim_{k\to \infty} X f_k(Z)x= \lim_{l\to \infty}\Phi(Z) f_k(Z)x=\Phi f(Z)x  \quad (Z\in \fB_d(n), x\in \C^n).
\end{align*}
Hence, $Xf=\Phi f$ for $f\in H_\cle(\clk)$. This shows that $\Phi f\in H_\clf(\clk)$ for all $f\in H_\cle(\clk)$, and thus $\Phi\in \text{Mult}_L(H_\cle(\clk), H_\clf(\clk))$. Thus, $X=M_{\Phi}$.
Converse follows from Proposition \ref{Left_mult_and_shift}. This completes the proof.
\end{proof}

We have a similar result for the intertwiners of left shift operators.
\begin{theorem}\label{intertwiner of left shifts}
Let $(M_{z_1} \otimes I_{\cle}, \ldots, M_{z_d}\otimes I_{\cle})$ and $(M_{z_1} \otimes I_{\clf}, \ldots, M_{z_d}\otimes I_{\clf})$ be the left shifts acting on $H_\cle(\clk)$ and $H_\clf(\clk)$, respectively. Suppose $X$ is a bounded operator from $H_\cle(\clk)$ into $H_\clf(\clk)$. Then $X \big (M_{z_i} \otimes I_{\cle}\big)= \big(M_{z_i} \otimes I_{\clf}\big) X$ $(i=1, \ldots, d)$ if and only if $X=M^R_\Phi$ for some $\Phi\in \text{Mult}_R(H_\cle(\clk), H_\clf(\clk))$.
\end{theorem}
\begin{proof}
Let $M_\cle=(M_{z_1} \otimes I_{\cle}, \ldots, M_{z_d}\otimes I_{\cle})$ and $M_\clf=(M_{z_1} \otimes I_{\clf}, \ldots, M_{z_d}\otimes I_{\clf})$. Since $X \big (M_{z_i} \otimes I_{\cle}\big)= \big(M_{z_i} \otimes I_{\clf}\big) X$ $(i=1, \ldots, d)$, we have 
\begin{align*}
X M^\alpha_\cle=M^\alpha_\clf X\quad (\alpha\in \mathbb{W}_d).
\end{align*}
Now, for $Z\in(Z_1, \ldots, Z_d)\in\fB_d(n)$ and $\eta\in \cle$,
\begin{align*}
(X M^\alpha_\cle) g_\eta(Z)=(X M^\alpha_\cle) (I_n\otimes \eta)(Z)=X (Z^{\alpha} \otimes \eta).
\end{align*}
On the other hand, for $Z\in\fB_d(n)$, $\eta\in \cle$, and $x\in \C^n$,
\begin{align*}
(M^\alpha_\cle X) g_\eta(Z)x &=(M^\alpha_\cle) (X g_\eta)(Z)x\\
&=(M^\alpha_\cle)\Phi(Z)(x\otimes\eta) \quad (\text{by the definition of $\Phi$})\\
&=\big[\Phi(Z)\bullet_R M^\alpha_\cle(Z)\big](x\otimes\eta)\\
&=M^R_\Phi(Z^\alpha x \otimes \eta)\\
&=M^R_\Phi(Z^\alpha  \otimes \eta)(x).
\end{align*}
Hence we have 
\[
X(Z^\alpha \otimes\eta)=M^R_\Phi(Z^\alpha  \otimes \eta).
\]
 It is easy to see that $\cls=\text{span } \{Z^\alpha\otimes\eta : Z\in\fB_d, \eta\in \cle, \alpha\in \mathbb{W}_d\}$ is dense in $H_\cle(\clk)$. Therefore,
\begin{align*}
[Xf](Z)=\big[M^R_\Phi f\big](Z) \quad (f\in \cls).
\end{align*}
Let $\{f_k\}_{k\geq 1}$ be a sequence in $\cls$ and that converges to $f$ in $H_\cle(\clk)$. Then by Remark \ref{Remark_f_n}, for each $Z\in \fB_d(n)$, $x\in \C^n$, $f_k(Z)x$ converges  $f(Z)x$ in  $\cle^n$. Clearly, $X f_k$ converges to $X f$, and thus, for each $Z\in \fB_d(n)$, $x\in \C^n$, $Xf_k(Z)x$ converges  $Xf(Z)x$ in  $\cle^n$. 
Now 
\begin{align*}
Xf(Z)x=\lim_{k\to \infty} X f_k(Z)x= \lim_{k\to \infty}\big[M^R_\Phi f_k\big](Z)x=\big[M^R_\Phi f\big](Z)x  \quad (Z\in \fB_d, x\in \C^n).
\end{align*}
Hence, $Xf=M^R_\Phi f$ for all $f\in H_\cle(\clk)$. Thus $\Phi\in \text{Mult}_R(H_\cle(\clk), H_\clf(\clk))$ and $X=M_{\Phi}$.
Converse follows from Proposition \ref{Right_mult_and_shift}. This completes the proof.
\end{proof}

\section{noncommutative commutant lifting theorem and noncommutative Nevanlinna-Pick interpolation}\label{lifting and interpolation}
In this section, we prove a commutant lifting theorem for nc Drury-Arveson space. Then, using commutant lifting, we will prove the nc Pick interpolation in the noncommutative ball.

Let $\Tilde{\clk}=\clk \otimes I_\cle$, where $\clk$ be the nc Szeg\H o kernel on $\fB_d$.
For a set $\Omega\subseteq \fB_d$, we define
\[\clh_\Omega(\Tilde{\clk}):=\overline{\text{span}}\,\{ \clk_{W, v, y}\otimes \eta : W\in \Omega_n, v, y\in \C^n: n\in \mathbb{N}, \eta \in \cle\}.\]
For $\Omega\subseteq \bM^d$, the \textit{nc envelop} of $\Omega$, denoted by $[\Omega]_{nc}$, is the smallest nc set containing the set $\Omega$.

We recall the following result from \cite{SSS}. For the sake of completeness, we provide a proof here.
\begin{lemma}\label{RKHS_full nc}
Let $\Omega\subseteq \fB_d$, and let $[\Omega]_{nc}$ be the nc envelope of $\Omega$. Then
\[\clh_\Omega(\Tilde{\clk})= \clh_{[\Omega]_{nc}}(\Tilde{\clk}).\]
\end{lemma}
\begin{proof}
 It is clear that $\clh_\Omega(\Tilde{\clk})\subseteq \clh_{[\Omega]_{nc}}(\Tilde{\clk})$, since $\Omega\subseteq [\Omega]_{nc}$. For other inclusion, let $f\in \clh_{[\Omega]_{nc}}(\Tilde{\clk}) $ such that $f\perp \clh_\Omega(\Tilde{\clk})$. Now for $W\in \Omega_n, v, y\in \C^n$, 
\begin{align*}
\la f, \clk_{W, v, y}\otimes \eta\ra =\la f(W)v, y\otimes \eta\ra=0.
\end{align*}
This shows that $f(W)=0$ for all $W\in \Omega$. Since $f$ is an nc function, therefore $f(W)=0$ for all $W\in [\Omega]_{nc}$. This completes the proof of the lemma.
\end{proof}

In the following, for $y=(y_1, \ldots, y_n)\in \cle^n$ and $c\in \C$, by $cy$ we mean $(cy_1, \ldots, cy_n)$.

\begin{lemma}\label{sum of nc kernel at points}
Let $K$ be a nc reproducing kernel on $\Omega$. Then
for $Z\in \Omega_n$, $u\in \cla^n$,  $x\in \cle^n$ and $W \in \Omega_m$, $v\in \cla^m$, $ y\in \cle^m$, and $c_1, c_2\in \C$, 
\[c_1K_{Z, u,x}+ c_2 K_{W, v, y}=K_{ Z\oplus W, u\oplus v, c_1x\oplus c_2y}.\]
\end{lemma}

\begin{proof}
Let $f\in H(K)$. Then
\begingroup
 \allowdisplaybreaks
\begin{align*}
\la f, K_{ Z\oplus W, u\oplus v, c_1x\oplus c_2y} \ra&=\la f(Z\oplus W)(u\oplus v), c_1x\oplus c_2y\ra_{\C^{n+m}}\\
&=\la (f(Z)\oplus  f(W))(u\oplus v), c_1x\oplus c_2y\ra_{\C^{n+m}}\\
&=\la f(Z) u, c_1x\ra_{\C^n} +   \la f(W) v,c_2y)\ra_{\C^{m}}\\
&=\Bar{c_1}\la f, K_{Z, u,x} \ra + \Bar{c_2}\la f,  K_{W, v, y}\ra\\
&=\la f, c_1K_{Z, u,x} + c_2K_{W, v, y}\ra.
\end{align*}
\endgroup
This completes the proof.
\end{proof}

A subspace $\clq\subseteq H_\cle(\clk)$ is said to be \textit{right shift co-invariant} if
\[\big (\Tilde{M}_{z_i}\otimes I_{\cle}\big)^*(\clq)\subseteq \clq \quad (i=1,\ldots, d).\]
 The following is a commutant lifting theorem corresponding to the right shifts.

\begin{theorem}\label{commutant_lifting theorem}
Let $\cle$ be a Hilbert space. Suppose $\clq$ is right shift co-invariant subspaces of $H_\cle(\clk)$, and $X\in \clb(\clq)$ with $\|X\|\leq 1$. If
\[X \big (P_{\clq} (\Tilde{M}_{z_i} \otimes I_{\cle})|_{\clq}\big)= \big(P_{\clq}(\Tilde{M}_{z_i} \otimes I_{\cle})_{\clq}\big) X \quad (i=1, \ldots, d),\]
then there exists a left multipliers $\Phi\in \text{Mult}_L(H_\cle(\clk))$ such that
\[\|M_\Phi\|=\|X\| \quad \text{ and }\quad X=P_{\clq} M_\Phi|_{\clq}.\]
\end{theorem}

\begin{proof}
It is clear that the noncommutative tuple $(P_{\clq} (\Tilde{M}_{z_1} \otimes I_{\cle})|_{\clq}, \ldots, P_{\clq} (\Tilde{M}_{z_d} \otimes I_{\cle})|_{\clq})$ dilates to $(\Tilde{M}_{z_1}\otimes I_{\cle}, \ldots, \Tilde{M}_{z_d}\otimes I_{\cle})$ on $H_{\cle}(\clk)$ via the dilation map $\Pi=\iota:\clq \to H_\cle(\clk)$, defined by $f\mapsto f$. That is, the isometry $\Pi$ satisfies 
\[
\Pi \big(P_{\clq} (\Tilde{M}_{z_1} \otimes I_{\cle})|_{\clq}\big )^*= \big(\Tilde{M}_{z_1} \otimes I_{\cle}\big )^*\Pi.  
\]
Due to Proposition \ref{right_shift_row}, the tuple $(\Tilde{M}_{z_1}\otimes I_{\cle}, \ldots, \Tilde{M}_{z_d}\otimes I_{\cle})$ is a row isometry. By \cite[Theorem 3.2]{P89}, there exists an operator $Y$ on $H_{\cle}(\clk)$ such that 
\[ X^*=Y^*|_{\clq},  \quad \|X\|=\|Y\|,\]
and 
\[Y  (\Tilde{M}_{z_i} \otimes I_{\cle})= (\Tilde{M}_{z_i} \otimes I_{\cle}) Y \quad (i=1, \ldots, d).\]
By Theorem \ref{intertwiner of right shifts}, there exists a left multipliers $\Phi\in \text{Mult}_L(H_\cle(\clk))$ such $Y=M_{\Phi}$. This completes the proof.
\end{proof}

The consequence of the above commutant lifting theorem is the following intertwining lifting result.

\begin{corollary}\label{intertwining_lifting}
Let $\cle$ and $\clf$ be two Hilbert spaces. Suppose $\clq_1$ and $\clq_2$ are right shift co-invariant subspaces of $H_\cle(\clk)$ and $H_\clf(\clk)$, respectively, and $X\in \clb(\clq_1, \clq_2)$ with $\|X\|\leq 1$. If
\[X \big (P_{\clq_1} (\Tilde{M}_{z_i} \otimes I_{\cle})|_{\clq_1}\big)= \big(P_{\clq_2}(\Tilde{M}_{z_i} \otimes I_{\clf})_{\clq_2}\big) X \quad (i=1, \ldots, d),\]
then there exists a left multiplier $\Phi\in \text{Mult}_L(H_\cle(\clk), H_\clf(\clk))$ such that $\|M_\Phi\|=\|X\|\leq 1$ and $X=P_{\clq_2} M_\Phi|_{\clq_1}$.
\end{corollary}

\begin{proof}
Let $\clq=\clq_1\oplus \clq_2$. We define operators $A$, $B_i$ $\in\clb(\clq)$, $i=1,\ldots, d$, as follows: 
\[ A=\begin{bmatrix} 0&0\\
X&0\end{bmatrix}\quad\text{and}\quad B_i=\begin{bmatrix} P_{\clq_1} (\Tilde{M}_{z_i} \otimes I_{\cle})|_{\clq_1}&0\\
0&P_{\clq_2} (\Tilde{M}_{z_i} \otimes I_{\clf})|_{\clq_2}\end{bmatrix}.\]
From the hypothesis, it is clear that 
\[A B_i=B_i A \quad (i=1,\ldots, d).\]
It is easy to see that $(B_1, \ldots, B_d)$ dilates to $(V_1,\ldots
, V_d)$ on $H_\cle(\clk)\oplus H_\clf(\clk)$, where
\[
V_i=\begin{bmatrix} \Tilde{M}_{z_i} \otimes I_{\cle} &0\\
0&\Tilde{M}_{z_i} \otimes I_{\clf}\end{bmatrix}.
\]
By the above theorem, there exists an operator $Y$ on $H_\cle(\clk)\oplus H_\clf(\clk)$ such that 
\begin{align}\label{id_commutant lifting}
A^*=Y^*|_{\clq},  \quad \|A\|=\|Y\|,\end{align}
and 
\begin{align}\label{id_bv}
Y  V_i= V_i Y \quad (i=1, \ldots, d).\end{align}
Let $Y=\begin{bmatrix} Y_{11} &Y_{12}\\
Y_{21}&Y_{22}\end{bmatrix}$. From \eqref{id_bv}, we have $Y_{21}(\Tilde{M}_{z_i} \otimes I_{\cle})= (\Tilde{M}_{z_i} \otimes I_{\clf}) Y_{21} $ for all $i=1, \ldots, d$. Therefore by \ref{intertwiner of right shifts}, there exists a left multiplier $\Phi\in \text{Mult}_L(H_\cle(\clk),H_\clf(\clk))$ such that $Y_{21}=M_\Phi$. Now the identity \eqref{id_commutant lifting} implies that 
\begin{align*}
A=\begin{bmatrix} 0&0\\
X&0\end{bmatrix}=P_\clq Y|_\clq &=  \begin{bmatrix} P_{\clq_1}&0\\
0&P_{\clq_2}\end{bmatrix}\begin{bmatrix} Y_{11} &Y_{12}\\
Y_{21}&Y_{22}\end{bmatrix}  \begin{bmatrix} P_{\clq_1}&0\\
0&P_{\clq_2}\end{bmatrix}= \begin{bmatrix} P_{\clq_1}Y_{11} P_{\clq_1} & P_{\clq_1}Y_{12}P_{\clq_1}\\
P_{\clq_1}Y_{21}P_{\clq_1}&P_{\clq_1}Y_{22}P_{\clq_1}\end{bmatrix}.
\end{align*}
This shows that $X=P_{\clq_2} M_\Phi|_{\clq_1}$ and $\|M_\Phi\|=\|X\|$. This completes the proof. 
\end{proof} 

\subsection{An application of the commutant lifting theorem to the noncommutative Nevanlinna-Pick interpolation:} We consider the following nc Nevanlinna-Pick interpolation problem.
Given $k$ distinct points $Z_1, \ldots, Z_k\in \fB_d$, and $W_1,\ldots, W_k\in \mathbb{M}^1$, we ask the following question: when does there exist an nc bounded free holomorphic function $\Phi\in H^\infty(\fB_d)$ such that 
\[\Phi(Z_i)=W_i \quad (i=1,\ldots, k) \quad \text{and }\quad \|\Phi\|_\infty\leq 1?\]
Since, being an nc set $\fB_d$ is closed under taking direct sum, by letting $Z_0=\oplus_{i=1}^k Z_i$ and $W_0=\oplus_{i=1}^k W_i$, the above $k$ points Nevanlinna-Pick problem turns into the following one point nc Nevanlinna-Pick problem on $\fB_d$: given point $Z_0\in \fB_d$, and $W_0\in \mathbb{M}^1$, when does there exist an nc holomorphic function $\Phi\in H^\infty(\fB_d)$ such that 
\[\Phi(Z_0)=W_0 \quad \text{and }\quad \|\Phi\|_\infty\leq 1?\]
Using the property that nc functions respect direct sums, solving the above one-point interpolation problem is equivalent to solving the preceding multi-point interpolation problem.

In the following theorem, we solve the above one-point interpolation problem by providing a necessary and sufficient condition, using the commutant lifting theorem established in the previous subsection. One may also consider a vector-valued version of the interpolation problem and solve it using the intertwining lifting result established in Corollary \ref{intertwining_lifting}. For the convenience of the reader, we treat the latter case separately in Theorem \ref{vector-valued Pick}.
\begin{theorem}\label{Scaler nc pick theorem}
Let $Z_0\in \fB_d(n)$ and $W_0\in M_n(\C)$. Then there exists $\Phi\in H_1^\infty(\fB_d)$ such that $\Phi(Z_0)=W_0$ if and only if
\[ \Tilde{K}(Z_0, Z_0)(\cdot):=\clk(Z_0, Z_0)(\cdot)-W_0 \clk(Z_0, Z_0)(\cdot) W_0^*\]
is a complete positive map on $M_n(\C)$, where $\clk$ is the nc Drury-Arveson kernel on $\fB_d$.
\end{theorem}
\begin{proof} First, we prove the sufficient part. 
 Let $\Omega=\{Z_0\}$. Suppose $[\Omega]_{nc}$ is the nc envelop of $\Omega$. We consider 
\[\clq:=\clh_{[\Omega]_{nc}}(\clk)=\overline{\text{span}}\,\{ \clk_{W, v, y}: W\in [\Omega]_{nc}(n), v, y\in \C^n, n\in \mathbb{N}\}\subseteq \clh^2_d.\]
Define an operator $X$ on $\clq$, through its adjoint, by 
\[
X^*\clk_{\oplus^k Z_0, \oplus_{i=1}^k v_i, \oplus_{i=1}^k y_i}=\sum_{i=1}^k\clk_{Z_0, v_i, W^*_0 y_i},
\]
where $\oplus^k Z_0$ is the direct sum of $k$ copies of $Z_0$ and for  $i=1, \ldots, k$,   $v_i, y_i\in \C^n$.
First, we show that $X$ is a contraction on $\clh_{[\Omega]_{nc}}(\clk)$. To this end, for each natural number $k$ and 
for $v_i, y_i\in \C^n$ ($i=1, \ldots, k$), we compute
\begingroup
 \allowdisplaybreaks
\begin{align*}
&\Big\langle (I_{\clq}-X X^*) \clk_{\oplus^k Z_0, \oplus_{i=1}^k v_i, \oplus_{i=1}^k y_i},  \clk_{\oplus^k Z_0, \oplus_{i=1}^k v_i, \oplus_{i=1}^k y_i}\Big\rangle\\
=& \Big\langle\clk_{\oplus^k Z_0, \oplus_{i=1}^k v_i, \oplus_{i=1}^k y_i},  \clk_{\oplus^k Z_0, \oplus_{i=1}^k v_i, \oplus_{i=1}^k y_i}\Big\rangle- \Big\langle X^*\clk_{\oplus^k Z_0, \oplus_{i=1}^k v_i, \oplus_{i=1}^k y_i},   X^* \clk_{\oplus^k Z_0, \oplus_{i=1}^k v_i, \oplus_{i=1}^k y_i}\Big\rangle\\
=& \Big\langle\sum_{i=1}^k\clk_{ Z_0, v_i, y_i}, \sum_{i=1}^k \clk_{ Z_0, v_i, y_i}\Big\rangle- \Big\langle \sum_{i=1}^k\clk_{Z_0, v_i, W^*_0 y_i},   \sum_{i=1}^k\clk_{Z_0, v_i, W^*_0 y_i}\Big\rangle\quad(\text{by Lemma \ref{sum of nc kernel at points}})\\
=& \Big\langle \big[\clk(Z_0, Z_0)(v_i v^*_j)\big]_{k\times k}  \big(\oplus_{i=1}^k y_i\big),  \oplus_{i=1}^k y_i\Big\rangle-  \Big\langle \big[W_0\clk(Z_0, Z_0)(v_i v^*_j)W^*_0\big]_{k\times k}  \big(\oplus_{i=1}^k y_i\big),  \oplus_{i=1}^k y_i\Big\rangle\\
=& \Big\langle\Big(\big[\clk(Z_0, Z_0)(v_i v^*_j)\big]_{k\times k} -\big[W_0\clk(Z_0, Z_0)(v_i v^*_j)W^*_0\big]_{k\times k} \Big ) \big(\oplus_{i=1}^k y_i\big),  \oplus_{i=1}^k y_i\Big\rangle\succeq0.
\end{align*}
\endgroup
The last positivity follows from the fact that $\Tilde{K}(Z_0, Z_0)(\cdot)$ is completely positive. This shows that $I_{\clq}-X X^*\geq 0$. In other words, $X$ is a contraction. 

Now, let $Z_0=(Z_{(0,1)} ,\ldots,Z_{(0,d)})$. Since, by Proposition \ref{action_adj_right_shift}, $\Tilde{M}^*_{z_i} \clk_{Z_0, v, y}=\clk_{Z_0, Z_{(0,i)} v, y}$ (for all $i=1,\ldots,d$), $\clq$ is right shift co-invariant. For $v, y\in \C^n$,
\begin{align*}
\Tilde{M}^*_{z_i}  X^*\clk_{Z_0, v, y}=\Tilde{M}^*_{z_i} \clk_{Z_0, v, W^*_0 y}= \clk_{Z_0, Z_{(0,i)} v, W^*_0 y}=X^*\clk_{Z_0, Z_{(0,i)} v,  y}=X^* \Tilde{M}^*_{z_i}\clk_{Z_0, v, y}.
\end{align*}
Therefore, $\Tilde{M}^*_{z_i}  X^*=X^*\Tilde{M}^*_{z_i}$  and thus
\[X (P_{\clq}\Tilde{M}_{z_i}|_{\clq})=(P_{\clq}\Tilde{M}_{z_i}|_{\clq}) X, \]
for all $i=1,\ldots, d$. By Theorem \ref{commutant_lifting theorem}, there exists a left multipliers $\Phi\in \text{Mult}_L(\clh^2_d)\cong H^{\infty}(\fB_d)$ such that
\[\|M_\Phi\|=\|X\| \quad \text{ and }\quad X=P_{\clq} M_\Phi|_{\clq}.\]
Let $f\in \clh^2_d$ be defined by $f(Z)=I_n$ for $Z\in \fB_d(n)$, for all $n$. Then,
 \begin{align*}
 \la f,X^*\clk_{Z_0, v, y} \ra=\la f, \clk_{Z_0, v, W^*_0 y}  \ra=\la I_n v, W^*_0 y\ra=\la  v, W^*_0 y\ra \quad (v, y\in \C^n).
 \end{align*}
On the other hand,
 \begin{align*}
 \la f,M_\Phi^*\clk_{Z_0, v, y} \ra=\la f, \clk_{Z_0, v, \Phi(Z_0)^* y}  \ra=\la I_n v, \Phi(Z_0)^* y\ra=\la  v, \Phi(Z_0)^* y\ra\quad (v,y \in \C^n).
 \end{align*}
This shows that 
\[\Phi(Z_0)=W_0.\]
Proof of the necessary part follows from the fact that $\Phi\in H_1^\infty(\fB_d)$ and $\Phi(Z_0)=W_0$.
\end{proof}

The following provides an example of $Z_0$ and $W_0$ satisfying the hypotheses of the above Theorem, and we explicitly construct the corresponding interpolant.
\begin{example}
We fix $d=1$ and consider
\[\fB_1=\{X\in M_n(\C): \|X\|<1: n\geq 1 \}.\]
We take 
\[Z_0=\begin{bmatrix} 0 & 1 \\ 0 & 0 \end{bmatrix}\in \fB_1 \quad \text{and}\quad   W_0 = \lambda(I+Z_0) = \begin{bmatrix} \lambda & \lambda \\ 0 & \lambda \end{bmatrix},\]
where $\lambda=\frac{2}{1+\sqrt{5}}$. Note that $\frac{1}{\lambda}=\frac{1+\sqrt{5}}{2}$ is a root of the equation $x^2-x-1=0$. This implies that 
\begin{align}\label{gold_id}
1-\lambda^2=\lambda.
\end{align}
\textbf{Claim:} The map $\Theta: M_2(\C) \to M_2(\C)$ defined by 
\[
\Theta(\cdot):=\clk(Z_0, Z_0)(\cdot)-W_0 \clk(Z_0, Z_0)(\cdot) W_0^* 
\]
is a \textit{cp} map on $M_2(\C)$.

Since, $Z^n_0= \begin{bmatrix}0&0\\0 &0\end{bmatrix}$ for all $n\geq 2$, for $P\in M_2(\C)$, we obtain
\allowdisplaybreaks
\begingroup
\begin{align*}
\Theta(P)& =P+ Z_0P Z^*_0- W_0 P W^*_0 -W_0 Z_0 P Z^*_0 W^*_0 \\
&= P + Z_0 P Z_0^* - W_0 P W_0^* - \lambda^2 Z_0 P Z_0^* \quad(\text{using the fact } W_0 Z_0=\lambda Z_0)\\
 & = P + (1-\lambda^2)\, Z_0 P Z_0^* - W_0PW_0^*\\
  &= P + \lambda Z_0 P Z_0^* - W_0 P W_0^* \hspace{5.8cm}(\text{by }(\ref{gold_id}))\\
  &=  P + \lambda Z_0 P Z_0^* - \lambda^2(I+Z_0) P (I+Z_0^*)\\
  &= P + \lambda Z_0 P Z_0^* - \lambda^2(P+Z_0^*+Z_0P+Z_0PZ_0^*)\\
  &= \lambda P + \lambda Z_0 P Z_0^* -\lambda^2(Z_0^*+Z_0P+Z_0PZ_0^*)\hspace{2.8cm}(\text{by }(\ref{gold_id}))\\
  &= \lambda (P +  Z_0 P Z_0^* -\lambda Z_0^*-\lambda Z_0P-\lambda Z_0PZ_0^*)\\
  &= \lambda \big((I-\lambda Z_0)P +  (1-\lambda) Z_0 P Z_0^* -\lambda Z_0^*\big)\\
  &=\lambda \big((I-\lambda Z_0)P +  \lambda^2 Z_0 P Z_0^* -\lambda Z_0^*\big)\hspace{4cm}(\text{by }(\ref{gold_id}))\\
&=\lambda \big((I-\lambda Z_0)P -  (I-\lambda Z_0)P\lambda Z_0^*\big)\\
&=\lambda (I-\lambda Z_0)P(I-\lambda Z_0^*)\\
&=APA^*,
  \end{align*}
  \endgroup
where $A=\sqrt{\lambda}(I-\lambda Z_0)\in M_2(\C)$. The above representation reveals that $\Theta$ is a \textit{cp} map. This proves the claim. Therefore, by Theorem \ref{Scaler nc pick theorem}, there exist $\Phi\in H^\infty_1(\fB_1)$ such that $\Phi(Z_0)=W_0$. To find such $\Phi$, we write
\allowdisplaybreaks
\begingroup
\begin{align*}
W_0=\lambda I +\lambda Z_0 & = \lambda I+(1-\lambda^2)Z_0\hspace{4cm}(\text{by }(\ref{gold_id}))\\
& = Z_0 +\lambda I -\lambda^2 Z_0\\
& = Z_0 +\lambda I -\lambda^2 Z_0-\lambda Z_0^2\hspace{2cm}(\text{since }Z_0^2=0)\\
& =(Z_0 +\lambda I)- \lambda Z_0(Z_0+\lambda I)\\
&= (Z_0 +\lambda I)(I-\lambda Z_0)\\
& =(Z_0+\lambda I)(I+\lambda Z_0)^{-1}\hspace{2cm} (\text{since }Z_0^n=0, n\geq 2).
\end{align*}
\endgroup
So, we can take 
\[
\Phi(z)=\frac{z+\lambda}{1+ \lambda z} \quad (\lambda\in \D).
\]
It is immediate that, $\Phi\in H^\infty_1(\fB_1)$ and $\Phi(Z_0)=W_0$.
\end{example}

Using Theorem \ref{Scaler nc pick theorem}, first we provide an example of $Z_0$ and $W_0$ for which there is no bounded nc holomorphic function $\Phi$ satisfying \[\Phi(Z_0)=W_0\quad\text{and}\quad\|\Phi\|_\infty\leq 1.\] 
\begin{example}
We fix $d=1$, consider $\fB_1$ and take
\[Z_0=\begin{bmatrix}0&1\\0&0\end{bmatrix} \quad \text{and} \quad W_0=\begin{bmatrix}1&0\\0&0\end{bmatrix}.\]
First we observe that $$Z^n_0= \begin{bmatrix}0&0\\0 &0\end{bmatrix}\quad (\forall\, n\geq 2).$$
Now, for $P= \begin{bmatrix}a&b\\c &d\end{bmatrix}\in M_2(\C)$,
\allowdisplaybreaks
\begingroup
\begin{align*}&\clk(Z_0, Z_0)(P)-W_0\clk(Z_0, Z_0)(P) W_0^* \\
\vspace{0.2cm}&=P+ Z_0P Z^*_0- W_0 P W^*_0 -W_0 Z_0 P Z^*_0 W^*_0 \\
\vspace{0.2cm}&= \begin{bmatrix}a&b\\c &d\end{bmatrix} +\begin{bmatrix}0&1\\0 &0\end{bmatrix} \begin{bmatrix}a&b\\c &d\end{bmatrix} \begin{bmatrix}0 &0\\1 &0\end{bmatrix}-\begin{bmatrix}1&0\\0 &0\end{bmatrix}\begin{bmatrix}a&b\\c &d\end{bmatrix} \begin{bmatrix}1&0\\0 &0\end{bmatrix}\\&\hspace{6.1cm}- \begin{bmatrix}1&0\\0 &0\end{bmatrix} \begin{bmatrix}0&1\\0 &0\end{bmatrix}\begin{bmatrix}a&b\\c &d\end{bmatrix} \begin{bmatrix}0&0\\1 &0\end{bmatrix}\begin{bmatrix}1&0\\0 &0\end{bmatrix}\\
&=\begin{bmatrix}0&b\\c &d\end{bmatrix}.
\end{align*}
\endgroup
Let $P$ be a positive matrix. Therefore, $b=\Bar{c}$. The above computation shows that, for any positive semi-definite matrix $P=\begin{bmatrix}a&b\\c &d\end{bmatrix}$, since $$\det \begin{bmatrix}0&\bar c\\c &d\end{bmatrix}=-|c|^2,$$  $\clk(Z_0, Z_0)(P)-W_0\clk(Z_0, Z_0)(P) W_0^*$ is always negative definite. Hence, by Theorem \ref{Scaler nc pick theorem}, there is no bounded free holomorphic function $\Phi$ on $\fB_1$ such that $\Phi(Z_0)=W_0$ and $\|\Phi\|_\infty\leq 1$.
\end{example}

In the following, we prove a vector-valued version of the nc Nevanlinna-Pick interpolation theorem.
\begin{theorem}\label{vector-valued Pick}
 Let $\cle$ and $\clf$ be two Hilbert spaces. Suppose $Z_0\in \fB_d$ and $W_0\in \mathbb{M}^1(\clb(\cle, \clf))$. Then there exists a left multiplier $\Phi\in \text{Mult}_L(H_\cle(\clk), H_\clf(\clk))$ such that $\Phi(Z_0)=W_0$ and $\|M_\Phi\|\leq 1$ if and only if
\[ \clk(Z_0, Z_0)(\cdot)\otimes I_\clf -W_0 \big(\clk(Z_0, Z_0)(\cdot)\otimes I_{\cle}\big) W_0^*\]
is a completely positive map from $M_n(\C)$ to $M_n(\clb(\clf))$.
\end{theorem}
\begin{proof}  Similar to the proof of Theorem \ref{Scaler nc pick theorem}, consider
\[\clq_1:=\clh_{[\Omega]_{nc}}(\clk\otimes I_{\cle})=\overline{\text{span}}\,\{ \clk_{W, v, y}\otimes \xi: W\in [\Omega]_{nc}(n), v, y\in \C^n, n\in \mathbb{N},\zeta\in\cle\}\subseteq H_\cle(\clk),\]
and 
\[\clq_2:=\clh_{[\Omega]_{nc}}(\clk\otimes I_{\clf})=\overline{\text{span}}\,\{ \clk_{W, v, y}\otimes \zeta: W\in [\Omega]_{nc}(n), v, y\in \C^n, n\in \mathbb{N},\xi\in\clf\}\subseteq H_\clf(\clk).\]
Denote $\Tilde{\clk}:=\clk\otimes I_{\cle}$. Define an operator $X:\clq_1\to\clq_2$, through its adjoint, by 
\[
X^*\Tilde{\clk}_{\oplus^k Z_0, \oplus_{i=1}^k v_i, \oplus_{i=1}^k y_i\otimes\xi_i}=\sum_{i=1}^k\Tilde{\clk}_{Z_0, v_i, W^*_0 (y_i\otimes\xi_i)},
\]
where $v_i, y_i\in \C^n$ and $\xi_i\in\clf$, for all $i=1, \ldots, k$. A similar computation as in Theorem \ref{Scaler nc pick theorem}, with the help of the identification $\Tilde{\clk}_{Z,v,y\otimes\xi}=\clk_{Z,v,y}\otimes\xi$, shows that
\begingroup
 \allowdisplaybreaks
\begin{align*}
&\Big\langle (I_{\clq_2}-X X^*) \Tilde{\clk}_{\oplus^k Z_0, \oplus_{i=1}^k v_i, \oplus_{i=1}^k y_i\otimes \xi_i},  \Tilde{\clk}_{\oplus^k Z_0, \oplus_{i=1}^k v_i, \oplus_{i=1}^k y_i\otimes\xi_i}\Big\rangle\\
=& \Big\langle \big[\clk(Z_0, Z_0)(v_i v^*_j)\otimes I_\clf\big]_{k\times k}  \big(\oplus_{i=1}^k y_i\otimes\xi_i\big),  \oplus_{i=1}^k y_i\otimes\xi_i\Big\rangle
\\&\hspace{3cm}- \Big\langle \big[W_0\big(\clk(Z_0, Z_0)(v_i v^*_j)\otimes I_{\cle}\big)W^*_0\big]_{k\times k}  \big(\oplus_{i=1}^k y_i\otimes\xi_i\big),  \oplus_{i=1}^k y_i\otimes\xi_i\Big\rangle\\
=& \Big\langle\big[\tilde{\clk}(Z_0, Z_0)(v_i v^*_j)\big]_{k\times k}  \big(\oplus_{i=1}^k y_i\otimes\xi_i\big),  \oplus_{i=1}^k y_i\otimes\xi_i\Big\rangle,
\end{align*}
\endgroup
where
\[
\Tilde{\clk}(Z_0,Z_0)(\cdot):=\clk(Z_0, Z_0)(\cdot)\otimes I_\clf -W_0 \big(\clk(Z_0, Z_0)(\cdot)\otimes I_{\cle}\big) W_0^*.
\]
Since, $\Tilde{\clk}(Z_0,Z_0)(\cdot)$ is a completely positive linear map from $M_n(\C)$ to $M_n(\clb(\clf))$, $I_{\clq_2}-X X^*\geq 0$, that is, $\|X\|\leq 1$. 

Let $Z_0=(Z_{(0,1)} ,\ldots,Z_{(0,d)})$. Then
\begin{align*}
\big(\Tilde{M}_{z_i}\otimes I_{\cle}\big)^*  X^*\Tilde{\clk}_{Z_0, v, y\otimes\xi}&=\Tilde{M}^*_{z_i} \Tilde{\clk}_{Z_0, v, W^*_0 (y\otimes\xi)}= \Tilde{\clk}_{Z_0, Z_{(0,i)} v, W^*_0 (y\otimes\xi)}\\&=X^*\Tilde{\clk}_{Z_0, Z_{(0,i)} v, (y\otimes\xi)}=X^* \big(\Tilde{M}_{z_i}\otimes I_{\clf}\big)^* \Tilde{\clk}_{Z_0, v, y\otimes\xi},
\end{align*}
and thus
\[X \big(P_{\clq_1}(\Tilde{M}_{z_i}\otimes I_\cle)|_{\clq_1}\big)=\big(P_{\clq_2}(\Tilde{M}_{z_i}\otimes I_\clf)|_{\clq_2}\big) X, \]
for all $i=1,\ldots, d$. By Corollary \ref{intertwining_lifting}, there exists $\Phi\in \text{Mult}_L(H_\cle(\clk), H_\clf(\clk))$ such that $\|M_\Phi\|=\|X\|\leq 1$ and $X=P_{\clq_2} M_\Phi|_{\clq_1}$. 

Fix $\zeta\in\cle$ and consider $f_{\zeta}\in H_\cle(\clk)$ defined by $f_{\zeta}(Z)v=v\otimes\zeta$ for all $Z\in \fB_d(n)$, for all $n$. Then,
 \begin{align*}
 \la f_\zeta,X^*\Tilde{\clk}_{Z_0, v, y\otimes \xi} \ra=\la f_\zeta, \Tilde{\clk}_{Z_0, v, W^*_0 (y\otimes\xi)}  \ra=\la f_{\zeta}(Z_0) v, W^*_0 (y\otimes\xi)\ra=\la  v\otimes \zeta, W^*_0 (y\otimes\xi)\ra,
 \end{align*}
for all $v, y\in \C^n$ and $\xi\in\cle$. On the other hand,
 \begin{align*}
 \la f_\zeta,M_{\Phi}^*\Tilde{\clk}_{Z_0, v, y\otimes \xi} \ra=\la f_\zeta, \Tilde{\clk}_{Z_0, v, \Phi(Z_0)^* (y\otimes\xi)}  \ra=\la f_{\zeta}(Z_0) v, \Phi(Z_0)^* (y\otimes\xi)\ra=\la  v\otimes \zeta, \Phi(Z_0)^* (y\otimes\xi)\ra,
 \end{align*}
 for all $v, y\in \C^n$ and $\xi\in\cle$. Now, varying $\zeta$ on $\cle$, we can conclude
\[\Phi(Z_0)=W_0.\]
Proof of the necessary part follows from the fact that $\Phi\in \text{Mult}_L(H_\cle(\clk), H_\clf(\clk))$, $\|M_\Phi\|\leq 1$, and $\Phi(Z_0)=W_0$.
\end{proof}

\subsection{Recovery of classical Nevanlinna-Pick theorem from Theorem \ref{Scaler nc pick theorem}.}\label{subsection_recovery NP theorem}
Now we show that the necessary and sufficient condition of the classical Nevanlinna-Pick theorem can be recovered from Theorem \ref{Scaler nc pick theorem}. Recall that given $k$ distinct points $z_1, \ldots, z_k$ in $\D$ and $w_1, \ldots, w_k$ in $\C$, the classical Nevanlinna-Pick interpolation problem asks the following: when does there exist a holomorphic function $\phi$ on $\D$ such that 
\[\phi(z_i)=w_i\quad (i=1,\ldots, k) \quad\text{and}\quad \|\phi\|_{\infty}=\sup_{z\in \D}|\phi(z)|\leq 1?\]
In particular, we take $d=1$ and consider
\[\fB_1=\{X\in M_n(\C): \|X\|<1: n\geq 1 \}.\]
It is clear that $\fB_1(1)=\D$. Consider
\[
Z_0=\begin{bmatrix}
z_1 & 0 & \cdots & 0 \\
0 & z_2 & \cdots & 0 \\
\vdots & \vdots & \ddots & \vdots \\
0 & 0 & \cdots & z_k
\end{bmatrix} \quad \text{and}\quad W_0=\begin{bmatrix}
w_1 & 0 & \cdots & 0 \\
0 & w_2 & \cdots & 0 \\
\vdots & \vdots & \ddots & \vdots \\
0 & 0 & \cdots & w_k
\end{bmatrix}.
\]
Now, for $P=[p_{ij}]_{k\times k}\in M_k$,
\allowdisplaybreaks
\begingroup
\begin{align*}
\Tilde{K}(Z_0, Z_0)(P):&=\clk(Z_0, Z_0)(P)-W_0 \clk(Z_0, Z_0)(P) W_0^*\\
&=\big[\clk(z_i ,z_j) p_{ij}\big]_{k\times k}- \big[w_i\clk(z_i ,z_j)\Bar{w}_j p_{ij}\big]_{k\times k}\\
&=\begin{bmatrix}\frac{1-w_i\Bar{w}_j}{1-z_i\Bar{z}_j}p_{ij}\end{bmatrix}_{k\times k}\\
&=\begin{bmatrix}\frac{1-w_i\Bar{w}_j}{1-z_i\Bar{z}_j}\end{bmatrix}_{k\times k} \odot P,
\end{align*}
\endgroup
where $``\odot"$ denotes the Schur product of two matrices.
Also, for each positive integer $m$ and block matrices 
\[
\mathbf{P}:=\begin{bmatrix}
P_{11} & P_{12} & \cdots & P_{1m} \\
P_{21} & P_{22} & \cdots & P_{2m}\\
\vdots & \vdots & \ddots & \vdots \\
P_{m1} & P_{m2} & \cdots & P_{mm}
\end{bmatrix}_{km\times km} \quad \text{and}\quad \mathbf{W}:=\begin{bmatrix}
W_0 & 0 & \cdots & 0 \\
0 & W_0 & \cdots & 0 \\
\vdots & \vdots & \ddots & \vdots \\
0 & 0 & \cdots & W_0
\end{bmatrix}_{km\times km},
\]
where $P_{ij}\in M_k$ for all $i,j=1\ldots,m$, we calculate
\allowdisplaybreaks
\begingroup
\begin{align*}
\Tilde{K}(\oplus^m Z_0, \oplus^m Z_0)(\mathbf{P}):&=\clk(\oplus^m Z_0, \oplus^m Z_0)(\mathbf{P})-\mathbf{W} \clk(\oplus^m Z_0, \oplus^m Z_0)(\mathbf{P}) \mathbf{W}^*\\
&= \big[\Tilde{K}(Z_0, Z_0)(P_{ij})\big]_{km\times km}\\
&=\bigg[\begin{bmatrix}\frac{1-w_i\Bar{w_j}}{1-z_i\Bar{z_j}}\end{bmatrix}_{k\times k} \odot P_{ij} \bigg]_{km\times km}\\
&=\bigg[\begin{bmatrix}\frac{1-w_i\Bar{w_j}}{1-z_i\Bar{z_j}}\end{bmatrix}_{k\times k}\otimes \mathbf{1}_m\bigg]\odot \mathbf{P},
\end{align*}
\endgroup
where $\mathbf{1}_m$ is the matrix of $m\times m$ order with all the entries $1$.
The above computation shows that the positivity of Pick matrix $\begin{bmatrix}\frac{1-w_i\Bar{w_j}}{1-z_i\Bar{z_j}}\end{bmatrix}_{k\times k}$ is equivalent to the completely positivity of $\Tilde{K}(Z_0, Z_0)(\cdot)$. Summarizing the above, therefore, we recover the following classical Nevanlinna-Pick theorem.
\begin{theorem}
Let $z_1,\ldots,z_n$ be distinct points in $\D$ and $w_1,\ldots,w_n$ be in $\D$. Then, there exists a bounded holomorphic function $\varphi$ on $\D$ such that,
\[
\|\varphi\|_{\infty}\leq 1\quad\text{and}\quad\varphi(z_i)=w_i\quad(i=1,\ldots,n)
\]
if and only if the matrix
\[
\begin{bmatrix}
	\frac{1-w_i\bar w_j}{1-z_i\bar z_j}
\end{bmatrix}_{i,j=1}^n
\]
is positive semi-definite.
\end{theorem}

\section{Shift invariant subspaces of the noncommutative Druary-Arveson space}\label{invariant subspace}
In this section, our goal is to characterize the shift-invariant subspaces of the nc Druary-Arveson space. For this purpose, we adapt the technique from \cite{Sarkar-I, Sarkar-II}. 

Recall that an operator tuple $T=(T_1,\ldots,T_d)$ is said to be a row contraction on a Hilbert space $\clh$ if 
\[
\sum_{i=1}^d T_iT_i^*\leq I.
\]
A row contraction $T=(T_1,\ldots,T_d)$ is said to be pure if 
\[
\text{SOT}-\lim_{k\to\infty}\sum_{\alpha\in\bW_d,|\alpha|=k}T^{\alpha}(T^{\alpha})^*=0.
\]
The following is an explicit dilation theorem that will be used in the sequel.
\begin{theorem}\label{dilation}
Let $T=(T_1,\ldots,T_d)$ be a pure row contraction on a Hilbert space $\clh$. We denote
\[
D_T:=\Big( I-\sum_{i=1}^d T_iT_i^*\Big)^{1/2}\quad \text{and}\,\,\cld=\overline{\text{ran}}\,D_T.
\]
We define $\Pi:H_{\cld}(\clk)\to\clh$  by
\begin{equation}\label{Pi}
\Pi(\clk_{W,v,y}\otimes\eta)=\sum_{\alpha\in \bW_d}\la y,W^\alpha v\ra T^\alpha D_T\eta.
\end{equation}
Then 
\begin{itemize}
\item[(i)] $\Pi$ is a co-isometry.
\item[(ii)] $(\Pi^*h)(Z)=\sum_{\alpha\in\bW_d}Z^{\alpha}\otimes D_T(T^{\alpha})^*h$.
\item[(iii)] $\Pi(M_{z_i}\otimes I_{\cld})=T_i\Pi$, for all $i=1,\ldots,d$.
\end{itemize}
\end{theorem}
\begin{proof}
We define $\Pi^*:\clh\to H_{\cld}(\clk)$ by
\[
(\Pi^*h)(Z)=\sum_{\alpha\in\bW_d}Z^{\alpha}\otimes D_T(T^{\alpha})^*.
\]
We calculate,
\begin{align*}
\|\Pi^*h\|^2
&=\Big\langle \sum_{\alpha\in \bW_d} Z^{\alpha}\otimes D_TT^{\alpha*}h, \sum_{\alpha\in \bW_d} Z^{\alpha}\otimes D_TT^{\alpha*}\Big\rangle\\
&= \sum_{\alpha\in \bW_d}\la T^{\alpha}D_T^2T^{\alpha*}h,h\ra\\
&= \Big\langle \sum_{\alpha\in \bW_d} T^{\alpha}D_T^2T^{\alpha*}h,h \Big\rangle.
\end{align*}
Now, since
\[
\sum_{\alpha\in \bW_d, |\alpha|=k} T^{\alpha}D_T^2T^{\alpha*}=\sum_{\alpha\in \bW_d, |\alpha|=k} T^{\alpha}T^{\alpha*}-\sum_{\alpha\in \bW_d, |\alpha|=k+1} T^{\alpha}T^{\alpha*},
\]
viewing $\sum_{\alpha\in \bW_d} T^{\alpha}D_T^2T^{\alpha*}$ as a telescopic sum, using the above, we can see
\begin{equation}\label{identity}
\sum_{\alpha\in \bW_d} T^{\alpha}D_T^2T^{\alpha*}= I-\lim_{k\to \infty}\sum_{\alpha\in \bW_d, |\alpha|=k} T^{\alpha}T^{\alpha*}=I,
\end{equation}
where, for the last equality we have used the fact that $T$ is pure. Therefore,
$\|\Pi^*h\|^2=\la h,h\ra=\|h\|^2$ and hence $\Pi^*$ is an isometry.

Now, we calculate the adjoint of $\Pi^*$. To this end,
\begin{align*}
\la \Pi^*h, \clk_{W,v,y}\otimes\eta\ra&=\Big\langle \sum_{\alpha\in \bW_d} Z^{\alpha}\otimes D_TT^{\alpha*}h, \Big (\sum_{\alpha\in \bW_d} \la y, W^\alpha v\ra Z^\alpha \Big )\otimes \eta \Big\rangle \\
&=\sum_{\alpha\in \bW_d}\la y, W^\alpha v \ra \la D_TT^{\alpha*}h, \eta\ra\\
&=\sum_{\alpha\in \bW_d} \big\langle h, \la y, W^\alpha v \ra T^{\alpha} D_T\eta\\
&= \Big\langle h, \sum_{\alpha\in \bW_d} \la y, W^\alpha v \ra T^{\alpha} D_T\eta
\Big\rangle. 
\end{align*}
Therefore, the adjoint of $\Pi^*$ is $\Pi$, which is given by the relation (\ref{Pi}). Now, for $Z\in\fB_d,\beta\in\bW_d,\eta\in\cld$ and $h\in\clh$,
\[
\la \Pi (Z^{\beta}\otimes \eta), h \ra=\Big\langle Z^{\beta}\otimes \eta, \sum_{\alpha\in\bW_d}Z^{\alpha}\otimes D_T(T^{\alpha})^*h  \Big\rangle= \la T^{\beta} D_T\eta, h\ra,
\]
and hence
\[
\Pi (Z^{\beta}\otimes \eta)=T^{\beta} D_T\eta.
\]
Therefore,
\begin{equation}\label{dilation_intertwiner}
\Pi (M_{z_i}\otimes I_{\cld})(Z^{\beta}\otimes \eta)=\Pi (Z^{i.\beta}\otimes \eta)= T^{i.\beta} D_T\eta=T_i T^{\beta} D_T\eta=T_i\Pi(Z^{\beta}\otimes \eta),
\end{equation}
which implies $\Pi (M_{z_i}\otimes I_{\cld})= T_i\Pi$, for all $i=1,\ldots,d$. This completes the proof.
\end{proof}

In the following result, we give a description of nontrivial  closed joint $(T_1,\ldots,T_d)$-invariant subspaces of a Hilbert space $\clh$, where $(T_1,\ldots,T_d)$ is a pure row contraction.
\begin{theorem}\label{description of S}
Let $T=(T_1,\ldots,T_d)$ be a pure row contraction on a Hilbert space $\clh$. A nontrivial closed subspace $\cls$ is a joint $T$ invariant subspace of $\clh$ if and only if there exists a Hilbert space $\cld$ and partial isometric operator $\Pi:H_{\cld}(\clk)\to\clh$ such that
\[
\Pi(M_{z_i}\otimes I_{\cld})=T_i\Pi,
\]
for all $i=1,\ldots,d$, and 
\[
\cls=\Pi(H_{\cld}(\clk)). 
\]
\end{theorem}
\begin{proof}
Let $T|_{\cls}=(T_1|_S,\ldots,T_d|_S)$. For all $s\in\cls$,
\[
\sum_{i=1}^dT_i|_S(T_i|_S)^*s=\sum_{i=1}^dT_iT_i^*s,
\]
which gives us that 
\[
\sum_{i=1}^dT_i|_S(T_i|_S)^*\leq I,
\]
that is, $T|_S$ is a row contraction. Also,
\begin{align*}
\text{SOT-}\lim_{k\to\infty}\sum_{\alpha\in\bW_d,|\alpha|=k}(T|_S)^{\alpha}P_S(T|_S)^{\alpha*}&=\text{SOT-}\lim_{k\to\infty}\sum_{\alpha\in\bW_d,|\alpha|=k}T^{\alpha}P_ST^{\alpha*}|_S\\&\leq \text{SOT-}\lim_{k\to\infty}\sum_{\alpha\in\bW_d,|\alpha|=k}T^{\alpha}(T^{\alpha})^*=0.
\end{align*}
Therefore, the row contraction $T|_S$ is pure. Define
\[
\cld=\overline{\text{ran}}\Big(I-\sum_{i=1}^d T_iT_i^*\Big)^{1/2}.
\]
Then, by Theorem \ref{dilation}, there exists a co-isometric map $\Pi_S:H_{\cld}(\clk)\to\cls$ such that
\[
\Pi_S(M_{z_i}\otimes I_{\cld})=T_i|_S\Pi_S,
\]
for all $i=1,\ldots,d$. Let $\iota_S:S\to\clh$ be the inclusion map. Define $\Pi:=\iota_S\Pi_S$. Then,
\[
\Pi(M_{z_i}\otimes I_{\cld})=\iota_S\Pi_S(M_{z_i}\otimes I_{\cld})=\iota_ST_i|_S\Pi_S=T_i\Pi_S=T_i\iota_S\Pi_S=T_i\Pi,
\]
for all $i=1,\ldots,d$. Also,
\[
\Pi\Pi^*=\iota_S\Pi_S\Pi_S^*\iota_S^*=P_S,
\]
where $P_S:\clh\to\cls$ is the orthogonal projection. Thus, $\Pi$ is a partial isometry and $S=\text{ran}\,\Pi=\Pi(H_{\cld}(\clk))$.
\end{proof}
We will use the above theorem to characterize the nontrivial closed joint  $M_{\z}=(M_{z_1},\ldots,M_{z_d})$-invariant subspaces of $\clh^2_d$. To this end, we need $M_{\z}$ to be a pure row contraction, which is verified in the following result.
\begin{proposition}\label{M_z_is_pure_row}
The tuple of left shifts $M_z=(M_{z_1},\ldots,M_{z_d})$ on $\clh^2_d$ is a pure row contraction. 
\end{proposition}
\begin{proof}
First part of the proof is similar to Proposition \ref{right_shift_row}. We claim that 
\[
\Big(I-\sum_{i=1}^dM_{z_i}M_{z_i}^*\Big)f=P_{\text{const.}}f\quad\big(f\in\clh^2_d\big),
\]
where $P_{\text{const.}}$ is the orthogonal projection of $\clh^2_d$ onto the set of all constant functions in $\clh^2_d$. To justify the claim, we calculate
\begin{align*}
\Big\langle \Big(I_n-&\sum_{i=1}^dM_{z_i}M_{z_i}^*\Big)\clk_{Z,u,x},\clk_{W,v,y}\Big\rangle\\
&=\la\clk(W,Z)(vu^*)x,y\ra-\sum_{i=1}^d\la\clk(W,Z)(vu^*)Z_i^*x,W_i^*y\ra\\
&=\Big\langle\Big[\clk(W,Z)(vu^*)-\sum_{i=1}^d W_i\clk(W,Z)(vu^*)Z_i^*\Big]x,y\Big\rangle\\
&=\la[vu^*]x,y\ra\\
&=\la\clk_{0_n,u,x},\clk_{W,v,y}\ra.
\end{align*}
So,
\[
\Big(I_n-\sum_{i=1}^d M_{z_i}M_{z_i}^*\Big)\clk_{Z,u,x}(W)=\clk_{0_n,u,x}(W)=\sum_{\alpha\in\bW_d}\langle x,0_n^{\alpha} u\rangle W^{\alpha}=\clk_{Z,u,x}(0_n).
\]
Therefore, through a limiting argument, we prove the claim. 

From the claim, it follows that
\[
\Big(I-\sum_{i=1}^dM_{z_i}M_{z_i}^*\Big)=P_{\text{const.}}\geq 0.
\]
Hence, $(M_{z_1},\ldots,M_{z_d})$ is a row contraction. 

To prove that the row contraction $M_z$ is pure, we first show that for all $i=1,\ldots,d$, $M_{z_i}$ are pure shift. Indeed,
\[
M_{z_i}^{*k}\clk_{W,v,y}(Z)=\clk_{W,v,W_i^*y}(Z)=\sum_{\alpha\in\bW_d}\la W_i^ky,W^{\alpha}\ra Z^{\alpha}\to 0\quad\text{as}\,\,k\to\infty
\]
because $W\in\fB_d(n)$, so $\|W_i\|<1$ and hence $W_i$ is a pure contraction. Therefore, by a limiting argument, $M_{z_i}^{*k}\to 0$ as $n\to\infty$. 

Finally, due to the above,
\[
\text{SOT-}\lim_{k\to\infty}\sum_{\alpha\in\bW_d,|\alpha|=k}M_z^{\alpha}(M_z^{\alpha})^*=0.
\]
Hence, the row contraction $M_z=(M_{z_1},\ldots,M_{z_d})$ is pure. 
\end{proof}
 
The next theorem, which is the main result of this section, is a corollary of Theorem \ref{description of S} and Theorem \ref{intertwiner of left shifts}.
\begin{theorem}\label{inv-sub-left-shift}
Let $M_z=(M_{z_1},\ldots,M_{z_d})$ be a tuple of left shifts on the nc Drury-Arveson space $\clh^2_d$ and $\cls$ be a non-trivial closed subspace of $\clh^2_d$. Then, $S$ is joint $M_z$-invariant if and only if there exists a Hilbert space $\cle$ and a partial isometric right multiplier $\Psi\in\text{Mult}_R(H_{\cle}(\clk),\clh^2_d)$ such that 
\[
\cls=\Psi H_{\cle}(\clk).
\]
\end{theorem}
\begin{proof}
By Theorem \ref{description of S}, there exists a Hilbert space $\cle$ and a partial isometric map $\Pi:H_{\cle}(\clk)\to \clh^2_d$ such that,
\[
\Pi(M_{z_i}\otimes I_{\cle})=M_{z_i}\Pi,
\]
for all $i=1,\ldots,d$, where $M_{z_i}\otimes I_{\cle}$ are the left shifts on $H_{\cle}(\clk)$, and 
\[
\cls=\Pi(H_{\cle}(\clk)). 
\]
Since, $\Pi$ is an intertwiner of the left shift operators, by Theorem \ref{intertwiner of left shifts}, $\Pi=M^R_{\Psi}$ for some $\Psi\in \text{Mult}_R(H_{\cle}(\clk),\clh^2_d)$. Hence, $\cls=\Psi H_{\cle}(\clk)$. 
\end{proof} 

\vspace{0.1in} \noindent\textbf{Acknowledgement:} The authors would like to express their sincere gratitude to Professor Jaydeb Sarkar for several fruitful and insightful discussions. The research of the first-named author is supported by DST-INSPIRE Faculty Fellowship No. DST/INSPIRE/Faculty/2023/\\IFA-23-MA-201. The second-named author's research is supported by NBHM Postdoctoral Fellowship No. 0204/27/(26)/2023/R\&D- II/11926.

\end{document}